\documentclass{amsart}
\usepackage{amssymb}
\usepackage[usenames,dvipsnames]{xcolor}
\definecolor{darkblue}{rgb}{0.0, 0.0, 0.55}
\usepackage[pagebackref,colorlinks,linkcolor=BrickRed,citecolor=OliveGreen,urlcolor=darkblue,hypertexnames=true]{hyperref}
\usepackage{xurl}

\usepackage{enumerate}

\newcommand\N{\mathbb N}

\newcommand\R{\mathbb R}
\newcommand\C{\mathbb C}

\newcommand\ii{\mathrm i}
\def\1{\mathbf 1}

\newcommand\al\alpha
\newcommand\be\beta
\newcommand\ga\gamma
\newcommand\la\lambda
\newcommand\de\delta
\newcommand\ep\varepsilon
\newcommand\ph\varphi
\newcommand\si\sigma
\newcommand\ta\tau
\newcommand\rh\varrho
\newcommand\io\iota
\newcommand\La\Lambda

\DeclareMathOperator\trunc{trunc}

\DeclareMathOperator\tr{tr}
\DeclareMathOperator\hur{hur}

\DeclareMathOperator{\Sym}{Sym}

\theoremstyle{definition}
\newtheorem{thm}{Theorem}[section]

\newtheorem{cor}[thm]{Corollary}
\newtheorem{pro}[thm]{Proposition}
\newtheorem{rem}[thm]{Remark}
\newtheorem{ex}[thm]{Example}
\newtheorem{lem}[thm]{Lemma}
\newtheorem{df}[thm]{Definition}
\newtheorem{notation}[thm]{Notation}

\author[D. Sawall]{David Sawall}
\address{Technische Universit\"at Dresden, Germany}
\email{david.sawall@tu-dresden.de}

\author[M. Schweighofer]{Markus Schweighofer}
\address{Universität Konstanz, Germany}
\email{markus.schweighofer@uni-konstanz.de}

\title{Spectrahedral relaxations and tightenings of rigidly convex sets}
\date{October 27, 2026}

\begin{document}

\begin{abstract}
    We construct a natural spectrahedral relaxation of a rigidly convex set (rcs) that depends only on the cubic part and on the degree of the defining real zero polynomial.
    For quadratic polynomials, we prove exactness and for a product of linear polynomials, we show that a higher order version of our relaxation is exact.
    We also establish a way of tightening the relaxation by moving the origin inside the rcs, and that the rcs is recovered by sweeping the origin over its
     interior. To obtain all these results we use a linear form that is closely connected to the zeros of a given real zero polynomial. We also prove that any
      rcs can be written as an affine linear slice of the cone of non-negative polynomials. Using this, we can define a hierarchy of tightenings that are
      spectrahedral shadows. We can show that this tightening is exact if the given polynomial admits a determinantal representation.
      This tightening is similar to a tightening of a hyperbolicity cone due to Kummer, Plaumann, and Vinzant.
\end{abstract}

\maketitle

\section*{Introduction}
A polynomial $p\in\R[x_1,\ldots,x_n]$ is said to be \emph{real zero} if it does not vanish at the origin and if every restriction to a real line through the origin is real-rooted. Given a real zero polynomial $p$, the \emph{rigidly convex set} of $p$ is the set
\begin{equation*}
    C(p)=\{a\in\R^n\mid p(ta)\neq0\text{ for all }t\in(0,1)\}.
\end{equation*}
G\aa rding \cite{gar} showed that it is indeed convex. There is also a closely connected homogeneous version of real zero polynomials, called \emph{hyperbolic polynomials}. Spectrahedra containing the origin in their interior are rigidly convex sets, and spectrahedral cones are hyperbolicity cones. Recently, OpenAI \cite{OpenAI} uploaded a counterexample to the Generalized Lax Conjecture by constructing a hyperbolic polynomial whose hyperbolicity cone is not a spectrahedral cone. See also \cite{OpenAI2} for a hyperbolicity cone that is not a linear projection of a spectrahedral cone. Helton and Vinnikov proved in 2007 that every real zero polynomial in two variables admits a determinantal representation \cite{hv}, in particular every rigidly convex set in the plane is a spectrahedron. Kummer and Netzer \cite{KummerNetzer} used artificial intelligence to generate a proof that the hyperbolicity cone of every strictly hyperbolic polynomial is a spectrahedral cone.

Given a real zero polynomial $p\in\R[x_1,\ldots,x_n]$, we will construct a spectrahedron $S(p)$ described by an $(n+1)\times(n+1)$ linear matrix polynomial that contains $C(p)$ (Theorem \ref{relaxation}). Our main tool, that we introduce in Subsections \ref{ssLForm} and \ref{ssLFormTrace}, is a linear form that encodes the zeros of $p$. If $C$ is the companion matrix of the homogenization of $p$ with entries in $\R[x_1,\ldots,x_n]$, then
\begin{equation*}
    L_p(x_1^{\alpha_1}\cdots x_n^{\alpha_n})=\frac{1}{(\alpha_1+\ldots+\alpha_n)!}\partial_{x_1}^{\alpha_1}\cdots\partial_{x_n}^{\alpha_n}\tr((-C)^{\alpha_1+\ldots+\alpha_n}).
\end{equation*}
In general, $L_p$ does not come from a measure but if $p$ is a product of linear polynomials, then $L_p$ is just a sum of Dirac measures where the atoms come from the zeros of the linear polynomials. In this case, we use this to construct a finitely converging hierarchy of spectrahedral relaxations of the polyhedron $C(p)$ (Theorem \ref{relaxationhierarchy}). If $p=\det(I_d+x_1A_1+\ldots+x_nA_n)$ is of degree $d$, then $L_p$ can be computed via the trace of a polynomial expression of the $A_1,\ldots,A_n$.

In Section \ref{sQuadraticExact}, we will show that the relaxation is exact for quadratic polynomials and in Section \ref{sTighenRelax}, we tighten the relaxation by showing (a strengthening of)
\begin{equation*}
    C(p)=\bigcap_{a\in\operatorname{int}C(p)}\bigl(S(p(x+a))+a\bigr).
\end{equation*}

In Section \ref{sTightening}, we show
\begin{equation*}
    C(p)=\{a\in\R^n\mid L_p((1+a^Tx)(y^Tx-z_1)^2\cdots(y^Tx-z_{d-1})^2)\geq0\text{ on }\R^n\times\R^{d-1}\},
\end{equation*}
where $x=(x_1,\ldots,x_n)$ and $y=(y_1,\ldots,y_n)$ and where we have extended $L_p$ to $\R[x_1,\ldots,x_n,y_1,\ldots,y_n,z_1,\ldots,z_{d-1}]$ (Theorem \ref{RigidlyConvexSetAsSliceOfNonNegPolys}). By strengthening the nonnegativity condition, we obtain a hierarchy of inner approximations that are spectrahedral shadows (linear projections of spectrahedra). If the real zero polynomial has a determinantal representation, then every layer of this tightening is exact (Theorem \ref{DeterminantalTightening}). In Subsection \ref{ssHermiteBezout} we establish a connection of our linear form to Hermite and Bézout matrices that we use to show exactness of our tightening for an affine shift of elementary symmetric polynomials in Subsection \ref{ssElSym}. Finally, in Subsection \ref{sectionKPV}, we discuss related work by Kummer, Plaumann, and Vinzant.

\subsection*{Some notations}

We will always denote by $x=(x_1,\ldots,x_n)$ an $n$-tuple of variables and by $t$ a single variable. For $\alpha\in\N_0^n$, we write
\begin{align*}
    |\alpha|&:=\alpha_1+\ldots+\alpha_n,\\
    x^\alpha&:=x_1^{\alpha_1}\cdots x_n^{\alpha_n},\text{ and}\\
    \binom{|\alpha|}{\alpha}&:=\frac{|\alpha|!}{\alpha_1!\cdots\alpha_n!}.
\end{align*}
For a polynomial $f$ and a vector $v$ with one coordinate for each variable of $f$, we write $D_vf$ for the directional derivative of $f$ in the direction $v$.

\section{Real zero polynomials and spectrahedra}

\subsection{Definition and examples}

The following definition stems from \cite[§2.1]{hv}.

\begin{df}\label{dfrz}
We say that $p\in\R[x]$ is a \emph{real zero polynomial} if for all $a\in\R^n$ and $\la\in\C$,
\[p(\la a)=0\implies\la\in\R.\]
\end{df}

\begin{rem}\label{nonzeroatorigin}
If $p\in\R[x]$ is a real zero polynomial, then $p(0)\ne0$.
\end{rem}

\begin{pro}\label{splitrz}
Let $p\in\R[x]$. Then $p$ is a real zero polynomial if and only if for each $a\in\R^n$, the univariate polynomial
\[p(ta)=p(ta_1,\ldots,ta_n)\in\R[t]\]
splits (i.e., is a product of a non-zero constant and polynomials of degree one) in $\R[t]$.
\end{pro}

\begin{ex}\label{quadraticrealzero}
Let $p\in\R[x]$ be a quadratic real zero polynomial with $p(0)=1$. Then $p$ can be uniquely written as
\[p=x^TAx+b^Tx+1\] with a symmetric matrix $A\in\R^{n\times n}$ and a vector $b\in\R^n$. For $a\in\R^n$ the univariate
quadratic polynomial $p(ta)=a^TAat^2+b^Tat+1$ splits in $\R[t]$ if and only if its discriminant $(b^Ta)^2-4a^TAa=a^T(bb^T-4A)a$ is
nonnegative. Hence $p$ is a real zero polynomial if and only if \[bb^T-4A\succeq0.\]
\end{ex}

The most obvious examples of real zero polynomials are products of linear polynomials that do not vanish at the origin. But this is just the special case where
all $A_i$ are diagonal matrices of the following more general example:

\begin{pro}\label{detexample}
Let $A_0,A_1,\ldots,A_n\in\C^{d\times d}$ be hermitian matrices such that $A_0$ is definite and
\[p=\det(A_0+x_1A_1+\ldots+x_nA_n),\]
then $p$ is a real zero polynomial.
\end{pro}

A \emph{spectrahedron} in $\R^n$ is a set of the form
\[\{a\in\R^n\mid A_0+a_1A_1+\ldots+a_nA_n\succeq0\},\]
where $A_0,\ldots,A_n$ are hermitian matrices of the same size, i.e., the solution set of a \emph{linear matrix inequality}.
The term was coined by Ramana and Goldman \cite{rg}. Spectrahedra generalize polyhedra, which correspond to diagonal matrices $A_i$, and they are
exactly the feasible sets of semidefinite programs, which can be solved efficiently by interior point methods \cite{vb}. This is the main motivation
for describing a convex set by a linear matrix inequality or, as we do in this article, for approximating it by spectrahedra. We refer to \cite{np} for an introduction.

\begin{rem}\label{realsymmetric}
In the definition of a spectrahedron, one may equivalently require the matrices $A_0,\ldots,A_n$ to be real symmetric: writing $A_i=B_i+\ii C_i$ with real symmetric $B_i$ and real skew-symmetric $C_i$, the reader will easily check that the hermitian matrix $A_0+a_1A_1+\ldots+a_nA_n$ is positive semidefinite if and only if the real symmetric matrix
$\left(\begin{smallmatrix}B&-C\\C&B\end{smallmatrix}\right)$ of twice the size is, where $B:=B_0+a_1B_1+\ldots+a_nB_n$ and $C:=C_0+a_1C_1+\ldots+a_nC_n$. We allow hermitian matrices throughout because
determinantal representations with hermitian matrices are easier to come by than those with real symmetric matrices, compare Theorem \ref{vinnikov} and Corollary \ref{vcor} below.
\end{rem}

\subsection{The Helton--Vinnikov theorem}

The following celebrated partial converse to Proposition \ref{detexample}
has been obtained in 2007 by Helton and Vinnikov \cite[Theorem 2.2, §4]{hv}.
For a short account of the long history of partial results, we refer to \cite[§8]{hv}.
A purely algebraic proof of this theorem has been given by Hanselka \cite[Theorem 2]{han}.

\begin{thm}[Helton and Vinnikov]\label{vinnikov}
If $p\in\R[x_1,x_2]$ is a real zero polynomial of degree $d$ with $p(0)=1$, then there exist symmetric $A_1,A_2\in\R^{d\times d}$ such that
\[p=\det(I_d+x_1A_1+x_2A_2).\]
\end{thm}

A weaker version of this theorem will be enough for most of our purposes.
Since this weaker version seems to be considerably easier to prove
(see mainly \cite{gkvw}, also \cite{pv} and \cite[§7]{han}), we state it here as a corollary.
In the following, we will always prefer to use the corollary instead of the theorem.

\begin{cor}[Helton and Vinnikov]\label{vcor}
If $p\in\R[x_1,x_2]$ is a real zero polynomial of degree $d$ with $p(0)=1$, then there exist hermitian matrices $A_1,A_2\in\C^{d\times d}$ such that
\[p=\det(I_d+x_1A_1+x_2A_2).\]
\end{cor}

\begin{ex}\label{hv2}
Consider a quadratic real zero polynomial $p\in\R[x_1,x_2]$ with $p(0)=1$. Write
\[p=x^TAx+b^Tx+1\] with a symmetric matrix $A\in\R^{2\times 2}$ and a vector $b\in\R^2$.
Write moreover \[A=\begin{pmatrix}a_{11}&a_{12}\\a_{12}&a_{22}\end{pmatrix}\qquad\text{and}\qquad
b=\begin{pmatrix}b_1\\b_2\end{pmatrix}.\]
Then $bb^T-4A\succeq0$ by Example \ref{quadraticrealzero}. Hence the leading principal minors $r:=b_1^2-4a_{11}$
and $s:=\det(bb^T-4A)$ are nonnegative. Suppose first that $r>0$ and consider the real symmetric matrices
\[A_1:=\frac12\begin{pmatrix}b_1-\sqrt r&0\\0&b_1+\sqrt r\end{pmatrix}\]
and
\[A_2:=\frac1{2r}
\begin{pmatrix}b_1^2b_2-b_1b_2\sqrt{r}-4 a_{11}b_2+4a_{12}\sqrt{r}&\sqrt{rs}\\
   \sqrt{rs}&b_1^2 b_2+b_1 b_2 \sqrt{r}-4 a_{11}b_2-4a_{12}\sqrt{r}\end{pmatrix}.
\]
Then one can easily verify that $p=\det(I_2+x_1A_1+x_2A_2)$. If $r=0$ but $bb^T-4A\ne0$, then an orthogonal change of
coordinates in $\R^2$ (which transforms a determinantal representation into a determinantal representation) reduces to the case $r>0$.
If $bb^T-4A=0$, then $p=(1+\frac12b^Tx)^2=\det(I_2+x_1\frac{b_1}2I_2+x_2\frac{b_2}2I_2)$.
\end{ex}

\subsection{Rigidly convex sets}

\begin{df}\label{dfrzs}
Let $p\in\R[x]$ be a polynomial.
Then we call
\[Z(p):=\{a\in\R^n\mid p(a)=0\}\]
the (real) \emph{zero set} defined by $p$.
\end{df}

\begin{df}\label{dfrcs}
Let $p\in\R[x]$ be a real zero polynomial.
Then we call
\[C(p):=\{a\in\R^n\mid\forall\la\in(0,1):p(\la a)\ne0\}\]
the \emph{rigidly convex set} defined by $p$.
\end{df}

A priori, it is not even clear that rigidly convex sets are convex. This was however already known to
G\aa rding. In the case where $p$ has a determinantal
representation of the kind considered in Proposition \ref{detexample} above, it is however easy to show that $C(p)$ is not only convex but even a spectrahedron:

\begin{pro}\label{rcsdet}
Let $d\in\N_0$, let $A_0,A_1,\ldots,A_n\in\C^{d\times d}$ be hermitian with $A_0\succ0$, and let
\[p=\det(A_0+x_1A_1+\ldots+x_nA_n).\] Then
\[C(p)=\{a\in\R^n\mid A_0+a_1A_1+\ldots+a_nA_n\succeq0\}\]
and
\[C(p)\setminus Z(p)=\{a\in\R^n\mid A_0+a_1A_1+\ldots+a_nA_n\succ0\}.\]
\end{pro}

\begin{thm}[G\aa rding \cite{gar}]\label{rzshift}
Let $p\in\R[x]$ be a real zero polynomial and $a\in C(p)\setminus Z(p)$. Then the shifted polynomial $p(x+a)$ is a real zero polynomial as well
and $C(p)=C(p(x+a))+a$.
\end{thm}

\subsection{Hyperbolic polynomials and hyperbolicity cone}
Hyperbolic polynomials are a homogeneous version of real zero polynomials and date back to G\aa rding. Hyperbolicity cones replace the notion of rigidly convex sets.

\begin{df}\label{dfhb}
Let $p\in\R[x]$ be homogeneous. We call a vector $e\in\R^n$ with $e\ne0$
a \emph{hyperbolicity direction} of $p$ if for all
$a\in\R^n$ and $\la\in\C$,
\[p(a-\la e)=0\implies\la\in\R.\]
In this case, we call $p$ \emph{hyperbolic with respect to $e$} or \emph{hyperbolic in direction $e$}.
We call $p$ \emph{hyperbolic} if it is hyperbolic with respect to some direction.
\end{df}

\begin{rem}\label{nonzeroate}
\begin{enumerate}[(a)]
\item If $p\in\R[x]$ is hyperbolic with respect to $e\in\R^n$, then $p(e)\ne0$.
\item If $p\in\R[x]$ is hyperbolic, then the number of variables $n$ is of course greater or equal to $1$ since $\R^0$ contains
only the zero vector which can never be a hyperbolicity direction by definition.
\item For homogeneous $p\in\R[x]$, hyperbolicity is equivalent to
\[p(\la a-e)=0\implies\la\in\R\]
for all $a\in\R^n$ and $\la\in\C$ because a polynomial is real-rooted if and only if its reciprocal is real-rooted.
\item From the previous item it follows immediately that if $p\in\R[x]$ is hyperbolic with respect to $e\in\R^n$, then $p(x+e)$ is real zero.
\end{enumerate}
\end{rem}

\begin{pro}\label{detexamplehyperbolic}
Let $A_1,\ldots,A_n\in\C^{d\times d}$ be hermitian matrices and $e\in\R^n$
such that $e_1A_1+\ldots+e_nA_n$ is definite. Then
\[p:=\det(x_1A_1+\ldots+x_nA_n)\in\R[x]\]
is hyperbolic (of degree $d$) in direction $e$.
\end{pro}

\begin{df}\label{dfhbc}
Let $p\in\R[x]$ be hyperbolic in direction $e$.
Then we call
\[C(p,e):=\{a\in\R^n\mid\forall\la\in\R:(p(a-\la e)=0\implies\la\ge0)\}\]
the \emph{hyperbolicity cone} of $p$ with respect to $e$.
\end{df}

Hyperbolicity cones are closed convex cones with nonempty interior. Closedness follows from the continuous dependence of the roots of
$\la\mapsto p(a-\la e)$ on $a$ (the degree of this polynomial is $\deg p$ for every $a$, with leading coefficient $\pm p(e)\ne0$), and convexity is
G\aa rding's theorem \cite{gar}, see also \cite[Theorem 2]{ren}. The interior of $C(p,e)$ is
\[\{a\in\R^n\mid\forall\la\in\R:(p(a-\la e)=0\implies\la>0)\},\]
which contains $e$ and is the connected component of $e$ in $\{a\in\R^n\mid p(a)\ne0\}$ \cite[Proposition 1]{ren}; the cone $C(p,e)$ is its closure.

\begin{pro}\label{hbrz}
Let $n\ge1$ and let $p\in\R[x]$ be homogeneous. Then $p$ is hyperbolic in direction of the first unit vector $u$ of $\R^n$
if and only if its dehomogenization \[q:=p(1,x_2,\ldots,x_n)\in\R[x_2,\ldots,x_n]\]
is a real zero polynomial.
\end{pro}

\section{The relaxation}

\subsection{The linear form associated to a polynomial}\label{ssLForm}

All power series are formal. We refer to \cite[§3]{god}, \cite[§6.1]{rob} and \cite[§I.2]{rui} for an introduction to power series.

\begin{df}\label{dftrunc}
Suppose $a_\al\in\R$ for all $\al\in\N_0^n$ and consider the power series
\[p=\sum_{\al\in\N_0^n}a_\al x^\al\in\R[[x]].\] Then we call for $d\in\N_0$, the polynomial
\[\trunc_dp:=\sum_{\substack{\al\in\N_0^n\\|\al|\le d}}a_\al x^\al\in\R[x]\] the \emph{truncation of $p$ at degree $d$}.
\end{df}

\begin{df}\label{dflogexp}
Let $p\in\R[[x]]$ be a power series.
\begin{enumerate}[(a)]
\item
If $p$ has constant coefficient $0$, then the power series
\[\exp p:=\sum_{k=0}^\infty\frac{p^k}{k!}\in\R[[x]]\]
is well-defined because all monomials appearing in $p^k$ have degree at least $k$ so that
\[\trunc_d \exp p=\trunc_d\sum_{k=0}^d\frac{p^k}{k!}\]
for all $d\in\N_0$. We call it the \emph{exponential} of $p$.
\item
If $p$ has constant coefficient $1$, then the power series
\[\log p:=\sum_{k=1}^\infty\frac{(-1)^{k+1}}k(p-1)^k\in\R[[x]]\]
is well-defined because all monomials appearing in $(p-1)^k$ have degree at least $k$ so that
\[\trunc_d\log p=\trunc_d\sum_{k=1}^d\frac{(-1)^{k+1}}k(p-1)^k\]
for all $d\in\N_0$. We call it the \emph{logarithm} of $p$.
\end{enumerate}
\end{df}

The following can be found for example in \cite[Lemma 4.1]{god},
\cite[§5.4.2, Proposition 2]{rob} or \cite[§6.1.3]{rob}.

\begin{pro}\label{logexp}
Consider the sets
\begin{align*}
A&:=\{p\in\R[[x]]\mid\trunc_0p=0\}\qquad\text{and}\\
B&:=\{p\in\R[[x]]\mid\trunc_0p=1\}.
\end{align*}
Then the following hold:
\begin{enumerate}[(a)]
\item $\exp\colon A\to B$ and $\log\colon B\to A$ are inverse to each other.
\item $\exp(p+q)=(\exp p)(\exp q)$ for all $p,q\in A$.
\item $\log(pq)=(\log p)+(\log q)$ for all $p,q\in B$.
\end{enumerate}
\end{pro}

\begin{proof}
(a) can be proven in two different ways: One way is to play it back to known facts about converging power series from calculus
\cite[§5.4.2, Proposition 2]{rob}. Note that the argument given in \cite[Lemma 4.1]{god} looks innocent but in reality needs
good knowledge of multivariate power series \cite[§I.1, §I.2]{rui}. The other way to prove it is by using formal composition and derivation of power series:
Deduce the result from the univariate case $n=1$ by
formally deriving the formal composition (in either way) of the univariate logarithmic and exponential power series \cite[§6.1.3]{rob}.

(b) is an easy calculation:
\begin{multline*}
\exp(p+q)=\sum_{k=0}^\infty\frac{(p+q)^k}{k!}=\sum_{k=0}^\infty\frac1{k!}\sum_{i=0}^k\binom kip^iq^{k-i}\\
=\sum_{k=0}^\infty\sum_{i=0}^k\frac{p^i}{i!}\frac{q^{k-i}}{(k-i)!}=(\exp p)(\exp q).
\end{multline*}

(c) follows easily from (a) and (b).
\end{proof}

\begin{df}\label{dfriesz}
Let $p\in\R[[x]]$ satisfy $p(0)\ne0$ and let $d\in\N_0$. We define the linear form $L_{p,d}$ on $\R[x]$ \emph{associated to $p$
with respect to the virtual degree $d$}
by specifying it on the monomial basis of $\R[x]$, namely by setting
\[L_{p,d}(1)=d\]
and by requiring the identity of formal power series
\[-\log\frac{p(-x)}{p(0)}=\sum_{\substack{\al\in\N_0^n\\\al\ne0}}\frac1{|\al|}\binom{|\al|}\al L_{p,d}(x^\al)x^\al\]
to hold. If $p\in\R[x]$, then we call $L_p:=L_{p,\deg p}$ the linear form \emph{associated to $p$}.
\end{df}

\begin{pro}\label{logprop}
\begin{enumerate}[(a)]
\item If $p,q\in\R[[x]]$ satisfy $p(0)\ne0\ne q(0)$ and $d,e\in\N_0$, then \[L_{pq,d+e}=L_{p,d}+L_{q,e}.\]
\item If $p,q\in\R[x]$ satisfy $p(0)\ne0\ne q(0)$, then \[L_{pq}=L_p+L_q.\]
\end{enumerate}
\end{pro}

\begin{proof}
Part (b) follows from (a) by observing that $\deg(pq)=\deg p+\deg q$ for all $p,q\in\R[x]$.
To prove (a), we suppose WLOG $p(0)=1=q(0)$ and thus $(pq)(0)=1$. 
By Definition \ref{dfriesz},
we then have to show that the following identity of formal power series holds:
\[-\log((pq)(-x))=-\log(p(-x))-\log(q(-x)).\]
This follows from Proposition \ref{logexp}(c).
\end{proof}

\subsection{Linear forms and traces}\label{ssLFormTrace}

\begin{lem}\label{convergence}
Let \[\sum_{k=1}^\infty a_kt^k\in\C[[t]]\qquad(a_1,a_2,\ldots\in\C)\] be a univariate power series with positive radius of convergence. Let
$p\in\C[x]$ be a polynomial with $p(0)=0$.
Then there is some $\ep>0$ such that for each $z\in\C^n$ with $\|z\|_\infty<\ep$, the series
\[\sum_{k=1}^\infty a_kp(z)^k\]
is absolutely convergent even after fully expanding $a_kp(z)^k$ into the obvious sum of $m^k$ many terms where $m$ is the number of monomials in p.
\end{lem}

\begin{proof}
We have to bound the finite partial sums of the absolute values of the individual terms from above (see for example \cite[Definition 8.2.4]{t1} or \cite[§I.1]{rui}).
Write $p=t_1+\ldots+t_m$ where each $t_i$ involves only one monomial of $p$. 
We have to find $\ep>0$ and $C\in\R$ such that for each $z\in\C^n$ with $\|z\|_\infty<\ep$ and for each
$\ell\in\N$, we have
\[\sum_{k=1}^\ell\sum_{i_1=1}^m\ldots\sum_{i_k=1}^m|a_kt_{i_1}(z)\dotsm t_{i_k}(z)|\le C.\]
WLOG $m>0$. Let $r>0$ denote the radius of convergence of the univariate power series.
Choose $\rh\in\R$ with $0<\rh<r$. As $p(0)=0$, each $t_i$ vanishes at the origin.
By continuity, we can choose $\ep>0$ such that \[|t_i(z)|\le\frac\rh m\] for all $i\in\{1,\ldots,m\}$ and $z\in\C^n$ with $\|z\|_\infty<\ep$.
Since a univariate power series converges absolutely inside the radius of convergence
\cite[Theorem 4.1.6(b)]{t2}, we can set $C:=\sum_{k=1}^\infty |a_k|\rh^k<\infty$.
For each $\ell\in\N$ and $z\in\C^n$ with $\|z\|_\infty<\ep$, we have
\begin{align*}
\sum_{k=1}^\ell\sum_{i_1=1}^m\ldots\sum_{i_k=1}^m|a_kt_{i_1}(z)\dotsm t_{i_k}(z)|&\le
\sum_{k=1}^\ell m^k|a_k|\left(\frac\rh m\right)^k=\sum_{k=1}^\ell |a_k|\rh^k\le C.
\end{align*}
\end{proof}

\begin{df}\label{hurwitz}
Let $A_1,\dots,A_n\in\C^{d\times d}$ and $\al\in\N_0^n$. The $\al$-\emph{Hurwitz product} of $A_1,\ldots,A_n$ is the matrix
that arises as follows: First, form all words in $n$ letters where the $i$-th letter appears exactly $\al_i$ times. Then turn
each word into a product of matrices by substituting $A_i$ for the $i$-th letter. Finally, sum up all matrices that arise in this way.
Formally, we can define it as
\[\hur_\al(A_1,\ldots,A_n):=\sum_{\substack{f\colon\{1,\ldots,|\al|\}\to\{1,\ldots,n\}\\\forall i\in\{1,\ldots,n\}:\#f^{-1}(i)=\al_i}}A_{f(1)}\dotsm A_{f(|\al|)}\in\C^{d\times d}.\]
In particular, $\hur_0(A_1,\ldots,A_n)=I_d$.
\end{df}

\begin{pro}\label{hur}
Suppose $d\in\N_0$ and $A_1,\ldots,A_n\in\C^{d\times d}$ are hermitian. Then
\[p:=\det(I_d+x_1A_1+\ldots+x_nA_n)\in\R[x]\]
and
\[L_{p,d}(x^\al)=\frac1{\binom{|\al|}\al}\tr(\hur_\al(A_1,\ldots,A_n))\]
for all $\al\in\N_0^n$.
\end{pro}

\begin{proof}
For $q\in\C[x]$, we denote by $q^*\in\C[x]$ the polynomial which arises from $q$ by applying the complex conjugation to the coefficients. We have
\begin{align*}
p^*&=(\det(I_d+x_1A_1+\ldots+x_nA_n))^*=(\det((I_d+x_1A_1+\ldots+x_nA_n)^T))^*\\
&=(\det(I_d+x_1A_1^T+\ldots+x_nA_n^T))^*=\det(I_d+x_1A_1^*+\ldots+x_nA_n^*)\\
&=\det(I_d+x_1A_1+\ldots+x_nA_n)=p
\end{align*}
and therefore $p\in\R[x]$. It is easy to see that Hurwitz products of hermitian matrices are again hermitian
and therefore have real diagonal entries and hence real trace.

It is clear that $L_{p,d}(1)=d=\tr(I_d)=\tr(\hur_0(A_1,\ldots,A_n))$.
By Definition \ref{dfriesz}, it remains to show that
\[(*)\qquad-\log(p(-x))=\sum_{\substack{\al\in\N_0^n\\\al\ne0}}\frac1{|\al|}
\tr(\hur_\al(A_1,\ldots,A_n))x^\al.\]
The real multivariate power series on both sides converge absolutely in a neighborhood of the origin in $\R^n$. For the left hand side this follows from
Lemma \ref{convergence} by recollecting terms belonging to the same monomial. For the right hand side, we argue as follows: The number of words of
length $k$ in $n$ letters is $n^k$.
If the entries of each $A_i$ are bounded in absolute value by $c>0$, then the entries of a product of the $A_i$
with $k$ many factors are bounded in absolute value by $d^{k-1}c^k$. Then the trace of such a product is bounded by $(dc)^k$. Hence we get
\[\sum_{\substack{\al\in\N_0^n\\|\al|=k}}\left|\frac1{|\al|}\tr(\hur_\al(A_1,\ldots,A_n))a^\al\right|\le(cdn)^k\|a\|_\infty^k\le\left(\frac12\right)^k
\]
for all $a\in\R^n$ with $\|a\|_\infty\le\frac1{2cdn}$.

By the identity theorem for multivariate real power series \cite[Proposition I.2.9]{rui}, it suffices to show that
both series in $(*)$ converge absolutely to the same value in a neighborhood of the origin in $\R^n$. It is a subtle issue that uses the substitution of convergent power series \cite[Proposition I.2.6]{rui} to show that for all $a$ in a neighborhood of the origin in $\R^n$, the left hand side of $(*)$
evaluated at $a$ (i.e., $(-\log(p(-x)))(a)$) equals $-\ln(p(-a))$ where $\log$ stands for the
operation on power series defined in Definition \ref{dflogexp} and $\ln$ for the usual real logarithm.
On the other hand, the right hand side of $(*)$ evaluates at $a$ from a small neighborhood of the
origin to
\[\sum_{\substack{\al\in\N_0^n\\\al\ne0}}\frac1{|\al|}
\tr(\hur_\al(a_1A_1,\ldots,a_nA_n))=\sum_{k=1}^\infty\frac1k\tr((a_1A_1+\ldots+a_nA_n)^k).\]
It now suffices to fix $a\in\R^n$ such that the hermitian matrix
\[B:=a_1A_1+\ldots+a_nA_n\in\C^{d\times d}\] is of operator norm strictly less than $1$ (or equivalently
has all eigenvalues in the open real interval $(-1,1)$) and to show that
\[-\ln(\det(I_d-B))=\sum_{k=1}^\infty\frac1k\tr(B^k).
\]
Since the operator norm is sub-multiplicative, the matrix
$C:=\sum_{k=1}^\infty\frac1k B^k\in\C^{d\times d}$ exists. Obviously, $C$ is hermitian and its eigenvalues, listed according to their algebraic multiplicity,  arise from the eigenvalues of $I_d-B$ by taking minus the
real logarithm. 
Since determinant and trace are the product and sum, respectively, of the eigenvalues counted with algebraic multiplicity, we thus get the result.
\end{proof}

\begin{df}
    Let $d\in\N_0$ and $p\in\R[x]$ be a polynomial of degree at most $d$ with $p(0)=1$. Then the \emph{degree-$d$ companion matrix of $p$} is the companion matrix $C\in\R[x]^{d\times d}$ of the degree-$d$ homogenization of $p$ which we consider as a univariate polynomial with coefficients in $\R[x]$. That is, if $p=1+p_1+\ldots+p_d$ with $p_i\in\R[x]_i$ for $i=1,\ldots,d$, then
    \begin{equation*}
        C=\begin{pmatrix}
            0&\cdots&\cdots&0&-p_d\\
            1&\ddots&&\vdots&\vdots\\
            0&\ddots&\ddots&\vdots&\vdots\\
            \vdots&\ddots&\ddots&0&\vdots\\
            0&\cdots&0&1&-p_1
        \end{pmatrix}\in\R[x]^{d\times d}.
    \end{equation*}
    Note that from linear algebra it is well-known that $P=\det(tI-C)$, where $P\in\R[t,x]$ is the degree-$d$ homogenization of $p$. Therefore, we have $p=\det(I-C)$. If $p(0)\neq0$ then the \textit{degree-$d$ companion matrix of $p$} is the degree-$d$ companion matrix of $\frac p{p(0)}$. If $d=\deg p$, then we call the degree-$d$ companion matrix the \emph{companion matrix of $p$}.
\end{df}

\begin{pro}\label{LFormDifferentiation}
    Let $d\in\N_0$ and let $p\in\R[x]$ be a polynomial of degree at most $d$ with $p(0)\ne0$. Let $C\in\R[x]^{d\times d}$ be its degree-$d$ companion matrix. Then
    \begin{equation*}
        L_{p,d}(x^\alpha)=\frac{1}{|\alpha|!}\partial^\alpha\tr((-C)^{|\alpha|})
    \end{equation*}
    for all $\alpha\in\N_0^n$.
\end{pro}
\begin{rem}
    In the previous proposition, $\tr((-C)^m)$ is a homogeneous polynomial of degree $m$. One can see this by observing the degree structure of the matrix $C$ and its powers.
\end{rem}
\begin{proof}[Proof of Proposition \ref{LFormDifferentiation}]
    If $d=0$, then $p$ is a non-zero constant and both sides vanish. Assume henceforth that $d\ge1$. Replacing $p$ by $\frac p{p(0)}$ does not change either side: the companion matrix is defined to be that of $\frac p{p(0)}$, and Definition \ref{dfriesz} gives $L_{p,d}=L_{q,d}$ for $q:=\frac p{p(0)}$. Assume therefore that $p(0)=1$. The case $\alpha=0$ is the equality $L_{p,d}(1)=d=\tr(I_d)$.

    We first show the identity
    \begin{equation}
        -\log(p(-x))=\sum_{m\ge1}\frac1m\tr((-C)^m)\tag{$\ast$}
    \end{equation}
    in $\R[[x]]$. By the previous remark, $\tr((-C)^m)$ is homogeneous of degree $m$, so the right hand side is a well-defined formal power series. Let $s$ be a new variable and consider the substitution homomorphism
    \[\io\colon\R[[x]]\to\R[x][[s]],\qquad f\mapsto f(sx),\]
    which sends a power series $\sum_{m\ge0}f_m$ with homogeneous components $f_m$ of degree $m$ to $\sum_{m\ge0}f_ms^m$. It is an injective ring homomorphism, and it commutes with the formal logarithm (defined on power series with constant coefficient $1$ in $\R[x][[s]]$ by the same series as in Definition \ref{dflogexp}), since it preserves constant coefficients and maps $(x)^k$ into $s^k\R[x][[s]]$ for all $k$. Applying $\io$ to both sides of $(\ast)$ and using the homogeneity of $\tr((-C)^m)$ once more, it therefore suffices to prove
    \begin{equation}
        -\log p(-sx)=\sum_{m\ge1}\frac{s^m}m\tr((-C)^m)\tag{$\ast\ast$}
    \end{equation}
    in $\R[x][[s]]$. From $\det(tI-C)=t^dp\left(\frac xt\right)$ we obtain the polynomial identity $p(sx)=\det(I-sC)$, hence $p(-sx)=\det(I+sC)$. Both sides of $(\ast\ast)$ have constant coefficient $0$ with respect to $s$, so it is enough to compare their derivatives with respect to $s$. The formal logarithm satisfies $(\log q)'=\frac{q'}q$, and Jacobi's formula $(\det M)'=\det(M)\tr(M^{-1}M')$ holds for every $M\in\R[x][[s]]^{d\times d}$ with invertible constant term. Together with $(I+sC)^{-1}=\sum_{j\ge0}(-s)^jC^j$ this gives
    \begin{multline*}\frac{d}{ds}\bigl(-\log\det(I+sC)\bigr)=-\tr\bigl((I+sC)^{-1}C\bigr)\\=-\sum_{m\ge1}(-s)^{m-1}\tr(C^m)=\sum_{m\ge1}s^{m-1}\tr((-C)^m),\end{multline*}
    which is the derivative of the right hand side of $(\ast\ast)$. This proves $(\ast\ast)$ and hence $(\ast)$.

    Now fix $\alpha\ne0$, let $m:=|\alpha|$ and write $\tr((-C)^m)=\sum_{|\beta|=m}\frac{\partial^\beta\tr((-C)^m)}{\beta!}\,x^\beta$. The coefficient of $x^\alpha$ on the right-hand side of $(\ast)$ is
    \[\frac1m\cdot\frac{\partial^\alpha\tr((-C)^m)}{\alpha!}.\]
    By Definition \ref{dfriesz}, the coefficient of $x^\alpha$ on the left-hand side is
    \[\frac1m\binom{m}{\alpha}L_{p,d}(x^\alpha)=\frac{(m-1)!}{\alpha!}L_{p,d}(x^\alpha).\]
    Comparing the two coefficients yields $L_{p,d}(x^\alpha)=\frac{\partial^\alpha\tr((-C)^m)}{m!}$.
\end{proof}

\begin{ex}\label{ex:LFormNotMeasure}
    One easily checks that $L_p(p^2)=-\frac23$ for $p=1-x_1^2-x_2^2$. This shows that $L_p$ need not come from a measure.
\end{ex}

\begin{notation}
    Let $q\in\R[x]$ be a polynomial. Then we define $q(\partial):=q(\partial_1,\ldots,\partial_n)$ to be the differential operator obtained by replacing $x_i$ by $\partial_i$ for all $i=1,\ldots,n$. We also write $\partial$ for the column vector of partial derivatives $\partial_1,\ldots,\partial_n$.
\end{notation}

\begin{pro}\label{LFormIsApolarInnerProduct}
    Let $d\in\N_0$ and let $p\in\R[x]$ be a polynomial of degree at most $d$ with $p(0)\ne0$. Let $q\in\R[x]_k$ be a homogeneous polynomial of degree $k$. Denote by $C$ the degree-$d$ companion matrix of $p$. Then
    \begin{equation*}
        L_{p,d}(q)=(k!)^{-1}q(\partial)\tr((-C)^k).
    \end{equation*}
\end{pro}
\begin{rem}
    In other words, $L_{p,d}(q)$ is the apolar inner product of $q$ and $\tr((-C)^k)$.
\end{rem}
\begin{proof}
    Write $f:=\tr((-C)^k)$ and $q=\sum_{|\alpha|=k}c_\alpha x^\alpha$.
    For every $\alpha$ with $|\alpha|=k$, Proposition \ref{LFormDifferentiation} gives
    \[L_{p,d}(x^\alpha)=\frac{\partial^\alpha f}{k!}.\]
    Hence
    \[L_{p,d}(q)=\sum_{|\alpha|=k}c_\alpha\frac{\partial^\alpha f}{k!}=\frac1{k!}q(\partial)f.\]
\end{proof}

\begin{pro}\label{RestrictLFormToLine}
    Let $d\in\N_0$ and let $p\in\R[x]$ be a polynomial of degree at most $d$ with $p(0)\ne0$. Let $a\in\R^n$ and $q\in\R[t]$. Denote by $C$ the degree-$d$ companion matrix of $p$. Then
    \begin{equation*}
        L_{p,d}(q(a^Tx))=\tr(q(-C(a))).
    \end{equation*}
\end{pro}
\begin{proof}
    By linearity, it suffices to treat $q=t^k$ for some $k\in\N_0$. The case $k=0$ is $L_{p,d}(1)=d=\tr(I_d)$. Now let $k\ge1$ and set $f:=\tr((-C)^k)$. By the remark after Proposition \ref{LFormDifferentiation}, $f$ is homogeneous of degree $k$, and every homogeneous polynomial of degree $k$ satisfies $f(a)=\frac{(a^T\partial)^kf}{k!}$. Proposition \ref{LFormIsApolarInnerProduct} therefore gives
    \[L_{p,d}((a^Tx)^k)=(k!)^{-1}(a^T\partial)^kf=f(a)=\tr((-C(a))^k).\]
\end{proof}

\begin{pro}\label{applyMatrixToLFrom}
    Let $m,n,d\in\N_0$ and $A\in\R^{m\times n}$. Let $p\in\R[x_1,\ldots,x_n]$ satisfy $p(0)\ne0$, and let $q$ be a polynomial in $m$ variables. Writing $y=(y_1,\ldots,y_m)$,
    \[L_{p(A^Ty),d}(q)=L_{p,d}(q(Ax)).\]
\end{pro}
\begin{proof}
    Replacing $p$ by $\frac p{p(0)}$ does not change either side, by Definition \ref{dfriesz}. Assume $p(0)=1$, and write $r:=p(A^Ty)\in\R[y]$. Both sides are linear in $q$ and agree on constants, since both linear forms send $1$ to $d$. It is therefore enough to show $L_{r,d}(y^\beta)=L_{p,d}((Ax)^\beta)$ for every $\beta\in\N_0^m\setminus\{0\}$.

    Let $z\in\R^m$ be treated as a constant vector. Definition \ref{dfriesz}, applied to $r$ and to $p$, together with $-\log r(-z)=-\log p(-A^Tz)$, gives
    \[\sum_{\beta\ne0}\frac1{|\beta|}\binom{|\beta|}{\beta}L_{r,d}(y^\beta)z^\beta
    =\sum_{\alpha\ne0}\frac1{|\alpha|}\binom{|\alpha|}{\alpha}L_{p,d}(x^\alpha)(A^Tz)^\alpha.\]
    Fix $k\ge1$. The multinomial theorem yields $(v^Tx)^k=\sum_{|\alpha|=k}\binom{k}{\alpha}v^\alpha x^\alpha$ for every $v\in\R^n$, and applying $L_{p,d}$ produces
    \[L_{p,d}((v^Tx)^k)=\sum_{|\alpha|=k}\binom{k}{\alpha}v^\alpha L_{p,d}(x^\alpha).\]
    Taking $v=A^Tz$ and dividing by $k$, the degree-$k$ part of the series identity becomes
    \[\sum_{|\alpha|=k}\frac1k\binom{k}{\alpha}L_{p,d}(x^\alpha)(A^Tz)^\alpha
    =\frac1k L_{p,d}((z^TAx)^k).\]
    Expanding $(z^TAx)^k=\sum_{|\beta|=k}\binom{k}{\beta}z^\beta(Ax)^\beta$ and applying $L_{p,d}$ once more, the right-hand side equals
    \[\sum_{|\beta|=k}\frac1k\binom{k}{\beta}L_{p,d}((Ax)^\beta)z^\beta.\]
    Comparing coefficients of $z^\beta$ with the degree-$k$ part of $-\log r(-z)$ gives the claim, since $\frac1k\binom{k}{\beta}\ne0$.
\end{proof}

\begin{pro}\label{LFormOfRidgePolyTimesLinearPoly}
    Let $p\in\R[x]$ be a real zero polynomial of degree at most $d$ with $p(0)=1$, $q\in\R[t]$ a univariate polynomial, and $a,b\in\R^n$. Let $A,B\in\C^{d\times d}$ be hermitian matrices with $p(sa+tb)=\det(I_d+sA+tB)$. Such matrices exist by Corollary \ref{vcor}, after padding by zero blocks if $\deg p(sa+tb)<d$, as in the proof of Theorem \ref{relaxation} below. Then
    \begin{equation*}
        L_{p,d}((1+a^Tx)q(b^Tx))=\tr((I_d+A)q(B)).
    \end{equation*}
    In particular, if $a=0$, we have
    \begin{equation*}
        L_{p,d}(q(b^Tx))=\tr(q(B)).
    \end{equation*}
\end{pro}
\begin{proof}
    By linearity of $L_{p,d}$, we can restrict ourselves to the case $q=t^m$ for some $m\in\N_0$. Using Propositions \ref{hur} and \ref{applyMatrixToLFrom}, we get
    \begin{align*}
        L_{p,d}((1+a^Tx)q(b^Tx))&\stackrel{\ref{applyMatrixToLFrom}}{=}L_{p(sa+tb),d}((1+s)q(t))\\
        &=L_{p(sa+tb),d}(t^m)+L_{p(sa+tb),d}(st^m)\\
        &\stackrel{\ref{hur}}{=}\tr(B^m)+\frac1{\binom{m+1}{(1,m)}}\tr(\hur_{(1,m)}(A,B))\\
        &\stackrel{(\ast)}{=}\tr(B^m)+\tr(AB^m)\\
        &=\tr((I_d+A)B^m),
    \end{align*}
    here, in $(\ast)$, the Hurwitz product runs over the $\binom{m+1}{(1,m)}$ words with one factor $A$ and $m$ factors $B$, and each such word has trace $\tr(AB^m)$ by cyclic invariance of the trace.
\end{proof}

\subsection{Relaxing rigidly convex sets}

\begin{df}\label{dfpencil}
Let $p\in\R[[x]]$ be a power series with $p(0)\ne0$ and $d\in\N_0$. Consider the symmetric matrices
\[
A_0:=
\begin{pmatrix}
L_{p,d}(1)&L_{p,d}(x_1)&\ldots&L_{p,d}(x_n)\\
L_{p,d}(x_1)&L_{p,d}(x_1^2)&\ldots&L_{p,d}(x_1x_n)\\
\vdots&\vdots&&\vdots\\
L_{p,d}(x_n)&L_{p,d}(x_1x_n)&\ldots&L_{p,d}(x_n^2)
\end{pmatrix}\in\R^{(n+1)\times(n+1)}\]
and
\[
A_i:=
\begin{pmatrix}
L_{p,d}(x_i)&L_{p,d}(x_ix_1)&\ldots&L_{p,d}(x_ix_n)\\
L_{p,d}(x_ix_1)&L_{p,d}(x_ix_1^2)&\ldots&L_{p,d}(x_ix_1x_n)\\
\vdots&\vdots&&\vdots\\
L_{p,d}(x_ix_n)&L_{p,d}(x_ix_1x_n)&\ldots&L_{p,d}(x_ix_n^2)
\end{pmatrix}\in\R^{(n+1)\times(n+1)}
\]
for $i\in\{1,\ldots,n\}$.
\begin{enumerate}[(a)]
\item We call the linear matrix polynomial
\[M_{p,d}:=A_0+x_1A_1+\ldots+x_nA_n\in\R[x]^{(n+1)\times(n+1)}\]
the \emph{pencil associated to $p$ with respect to the virtual degree $d$} and
\[S_d(p):=\{a\in\R^n\mid M_{p,d}(a)\succeq0\}\]
the \emph{spectrahedron associated to $p$ with respect to the virtual degree $d$}.
\item In the case where $p$ is a polynomial, we call
\[M_p:=M_{p,\deg p}\] the \emph{pencil associated to $p$} and \[S(p):=S_{\deg p}(p)\]
the \emph{spectrahedron associated to $p$}.
\item
We call the linear matrix polynomial
\[M_{p,\infty}\in\R[x]^{n\times n}\] that arises from $M_{p,d}$ (for no matter what $d\in\N_0$)
by deleting the first row and column the \emph{pencil associated to $p$ with respect to infinite virtual degree}
and \[S_\infty(p):=\{a\in\R^n\mid M_{p,\infty}(a)\succeq0\}\]
the \emph{spectrahedron associated to $p$ with respect to infinite virtual degree}.
\end{enumerate}
\end{df}

\begin{rem}\label{dependsonlyoncubicpart}
Let $p\in\R[x]$ be a polynomial with $p(0)\ne0$.
Note that $M_p$ and therefore $S(p)$ depend only on the cubic part $\trunc_3p$ of $p$ and on the degree $\deg p$, the latter
entering only through the top left entry $L_p(1)=\deg p$ of $M_p$. If $p(0)=1$, then all other entries of $M_p$ are polynomials of degree at most $3$
in the coefficients of $\trunc_3p$, which could be written down explicitly by the expressions of Proposition \ref{LFormDifferentiation} for the values
of $L_p$ on the monomials of degree at most $3$.
\end{rem}

\begin{lem}\label{restriction}
Suppose $m,n\in\N_0$ with $m\le n$ and $q\in\R[x]$ with $q(0)\ne0$. Set
\[r:=q(x_1,\ldots,x_m,0,\ldots,0)\in\R[x_1,\ldots,x_m].\]
Then
\begin{enumerate}[(a)]
\item $L_{q,d}(p)=L_{r,d}(p)$ for all $p\in\R[x_1,\ldots,x_m]$ and $d\in\N_0$,
\item $\{a\in\R^m\mid(a,0,\ldots,0)\in C(q)\}=C(r)$ if $q$ is real zero,
\item $\{a\in\R^m\mid(a,0,\ldots,0)\in S_d(q)\}\subseteq S_d(r)$ for all $d\in\N_0\cup\{\infty\}$.
\end{enumerate}
\end{lem}
\begin{proof}
(a) Both sides are linear in $p$ and send $1$ to $d$. For the remaining monomials, Definition \ref{dfriesz} writes the coefficients of $-\log\frac{q(-x)}{q(0)}$ as $\frac1{|\alpha|}\binom{|\alpha|}{\alpha}L_{q,d}(x^\alpha)$. Substituting $x_{m+1}=\cdots=x_n=0$ produces $-\log\frac{r(-x)}{r(0)}$, since $r(0)=q(0)$. Comparing coefficients of the monomials in $x_1,\ldots,x_m$ gives the claim.

(b) Suppose that $q$ is real zero. Every univariate restriction of $r$ to a line through the origin is a univariate restriction of $q$, so $r$ is real zero. For $a\in\R^m$ one has $q(t(a,0,\ldots,0))=r(ta)$, hence $(a,0,\ldots,0)\in C(q)$ if and only if $a\in C(r)$.

(c) Let $d\in\N_0$, $a\in\R^m$ and $\tilde a:=(a,0,\ldots,0)\in\R^n$. Write $x_0:=1$. For $v=(v_0,\ldots,v_n)^T\in\R^{n+1}$ with $v_{m+1}=\cdots=v_n=0$,
\begin{align*}
v^TM_{q,d}(\tilde a)v
&=\sum_{i,j=0}^n v_iv_j\Bigl(L_{q,d}(x_ix_j)+\sum_{k=1}^n\tilde a_k L_{q,d}(x_kx_ix_j)\Bigr)\\
&=L_{q,d}\bigl((v_0+v_1x_1+\cdots+v_mx_m)^2(1+a_1x_1+\cdots+a_mx_m)\bigr).
\end{align*}
The polynomial on the right lies in $\R[x_1,\ldots,x_m]$. By (a) its value under $L_{q,d}$ equals its value under $L_{r,d}$, and that value is $v'^TM_{r,d}(a)v'$ for $v':=(v_0,\ldots,v_m)^T$. Thus $M_{q,d}(\tilde a)\succeq0$ implies $M_{r,d}(a)\succeq0$.

For the infinite pencil, delete the first row and column. The principal submatrix of $M_{q,\infty}(\tilde a)$ on the coordinates $x_1,\ldots,x_m$ equals $M_{r,\infty}(a)$, because those entries have positive degree and agree by (a). A principal submatrix of a positive semidefinite matrix is positive semidefinite.
\end{proof}

\begin{thm}\label{relaxation}
Let $p$ be a real zero polynomial and let $e$ be an integer with $e\ge\deg p$. Then $C(p)\subseteq S_e(p)$. In particular, $C(p)\subseteq S(p)$.
\end{thm}
\begin{proof}
    Replacing $p$ by $\frac p{p(0)}$ changes neither $C(p)$ nor $L_{p,e}$. Assume $p(0)=1$ and set $d:=\deg p$. If $d=0$, then $p=1$, so $L_{p,e}$ vanishes on every monomial of positive degree and $L_{p,e}(1)=e$. The pencil $M_{p,e}$ is therefore $e$ times the matrix with a one in the upper left corner and zeros elsewhere, which is positive semidefinite, and $C(p)=\R^n$. Assume henceforth $d\ge1$, and let $a\in C(p)$. For $v=(v_0,\ldots,v_n)^T\in\R^{n+1}$, with the convention $x_0:=1$,
    \[v^TM_{p,e}(a)v=L_{p,e}\bigl((1+a^Tx)(v_0+v_1x_1+\cdots+v_nx_n)^2\bigr).\]
    Set $b:=(v_1,\ldots,v_n)^T$ and $q:=(v_0+t)^2\in\R[t]$. The polynomial $r:=p(sa+tb)\in\R[s,t]$ is real zero and satisfies $r(0)=1$. Let $d':=\deg r$. If $d'=0$, take hermitian matrices $A=B=0\in\C^{e\times e}$. If $d'\ge1$, Corollary \ref{vcor} provides hermitian $A',B'\in\C^{d'\times d'}$ with $r=\det(I_{d'}+sA'+tB')$, and padding by zero blocks produces hermitian $A,B\in\C^{e\times e}$ with $r=\det(I_e+sA+tB)$. Proposition \ref{LFormOfRidgePolyTimesLinearPoly}, applied with virtual degree $e$, yields
    \[L_{p,e}((1+a^Tx)q(b^Tx))=\tr((I_e+A)(v_0I_e+B)^2).\]
    Every eigenvalue of $A$ is at least $-1$: an eigenvalue $\mu<-1$ would give $\det(I_e+\lambda A)=p(\lambda a)=0$ for $\lambda=-\frac1\mu\in(0,1)$, contradicting $a\in C(p)$. Thus $I_e+A\succeq0$. The hermitian square $(v_0I_e+B)^2$ is positive semidefinite as well, so
    \[\tr((I_e+A)(v_0I_e+B)^2)=\tr\bigl((I_e+A)^{1/2}(v_0I_e+B)^2(I_e+A)^{1/2}\bigr)\ge0.\]
    Therefore $M_{p,e}(a)\succeq0$, that is, $a\in S_e(p)$.
\end{proof}

\begin{rem}
    Example \ref{ex:LFormNotMeasure} shows that if one allows for higher degree moments in the matrix polynomial defining $S(p)$, one does not always get a relaxation.
\end{rem}

\subsection{Relaxing linear programs}

We now come back to the most basic example of real zero polynomials, namely products of linear polynomials non-vanishing at the origin.
The rigidly convex
sets they define are exactly the polyhedra containing the origin in their interior. The complexity behavior of optimization of linear functions over polyhedra is
mainly governed by the number of linear inequalities they are defined by \cite{mg}.
For polyhedra with a huge number of facets, it is therefore reasonable to try to
find reasonable outer approximations defined by a small linear matrix inequality. This fits of course into the above more general framework.
However, in this special case, we get new interpretations, simplifications and extensions of the construction presented above:

First, we will be able to
interpret the matrix coefficients $A_i$ of the pencil from Definition \ref{dfpencil} as moment matrices and localization matrices \cite{lau}. This might remind
the reader of Lasserre's moment relaxations. But Lasserre's relaxation is a lift-and-project method where
``moment matrices'' are actually matrices filled with unknowns having the structure of moment matrices. Our method is not a lift-and-project method
and the matrices are actual moment matrices filled with real numbers.

Second, the proofs for the case of linear programming will simplify dramatically. In particular, we will not need any version of the Helton--Vinnikov Theorem
\ref{vinnikov}.

Third, we will not restrict to moments of degree three in this case but will go to arbitrarily high moments and thus present a hierarchy of relaxations
for which we can prove finite convergence at level $d-1$ if $d$ is the number of linear inequalities.

\begin{pro}\label{dirac}
Let $d\in\N_0$ and $a_1,\dots,a_d\in\R^n$ and $p:=\prod_{i=1}^d(1+a_i^Tx)$. Then
\[L_{p,d}(q)=\sum_{i=1}^dq(a_i)\]
for all $q\in\R[x]$, i.e., $L_{p,d}$ is integration with respect to the sum of the Dirac measures in the points $a_i$.
\end{pro}

\begin{proof}
In the case $d=0$, we have $p=1$ and thus $L_{p,d}=0$. The case $d\ge2$ reduces to the case $d=1$ by means of Proposition
\ref{logprop}(a). Now let $d=1$ and write $a:=a_1$, so that $p=1+a^Tx$. By linearity, it suffices to treat the case where $q$ is a monomial. For the constant monomial $q=1$, we have $L_{p,1}(q)=1=q(a)$. For every other monomial, Definition \ref{dfriesz} and the expansion $-\log p(-x)=\sum_{k\ge1}\frac{(a^Tx)^k}k$ give
\[\sum_{k=1}^\infty\frac1k\sum_{|\alpha|=k}\binom{k}{\alpha}L_{p,1}(x^\alpha)x^\alpha=\sum_{k=1}^\infty\frac{(a^Tx)^k}k.\]
The multinomial theorem expands the right-hand side as $\sum_{k=1}^\infty\frac1k\sum_{|\alpha|=k}\binom{k}{\alpha}a^\alpha x^\alpha$. Comparing coefficients yields $L_{p,1}(x^\alpha)=a^\alpha$, and therefore $L_{p,1}(q)=q(a)$.
\end{proof}

The following generalizes Definition \ref{dfpencil}(b).

\begin{df}\label{dfpencilhierarchy}
Let $p\in\R[x]$ with $p(0)\ne0$ and $d\in\N_0$. Set $s:=\binom{d+n}d=\binom{d+n}n=\#E$ where \[E=\{x^\al\mid \al\in\N_0^n,|\al|\le d\}\]
is the set of monomials in $n$ variables of degree at most $d$.
Fix an order on these monomials, i.e., write $E=\{m_1,\ldots,m_s\}$.
Consider the symmetric matrices
\[A_0:=(L_p(m_im_j))_{\substack{1\le i\le s\\1\le j\le s}}\qquad\text{and}\qquad A_k:=(L_p(x_km_im_j))_{\substack{1\le i\le s\\1\le j\le s}}\]
for $k\in\{1,\ldots,n\}$. Then we call the linear matrix polynomial
\[M_p^{(d)}:=A_0+x_1A_1+\ldots+x_nA_n\in\R[x]^{s\times s}\]
the \emph{$d$-th pencil associated to $p$} and
\[S^{(d)}(p):=\{a\in\R^n\mid M_p^{(d)}(a)\succeq0\}\]
the \emph{$d$-th spectrahedron associated to $p$}.
\end{df}

\begin{lem}\label{pencilevalhierarchy}
Let $p\in\R[x]$ be a polynomial with $p(0)\ne0$, $a\in\R^n$ and $d\in\N_0$. Let $m_1,\ldots,m_s$ be the pairwise distinct monomials
in $n$ variables of degree at most $d$ in the order that has been fixed in Definition \ref{dfpencilhierarchy}. Let $v\in\R^s$.
Then \[v^TM_p^{(d)}(a)v=L_p((v_1m_1+\ldots+v_sm_s)^2(1+a_1x_1+\ldots+a_nx_n)).\]
\end{lem}

\begin{lem}\label{interpol}
Let $d\in\N$, $a_1,\ldots,a_d\in\R^n$ and $i\in\{1,\ldots,d\}$. Then there exists a polynomial $q\in\R[x]$ with $\deg q<d$ such that $q(a_i)\ne0$ and $q(a_j)=0$ for all $j\in\{1,\ldots,d\}$ with $a_j\ne a_i$.
\end{lem}

\begin{proof}
Set $J:=\{j\in\{1,\ldots,d\}\mid a_j\ne a_i\}$. The polynomial $\prod_{j\in J}(x^T(a_j-a_i))\in\R[x]$ is not the zero polynomial and hence cannot vanish on the whole of $\R^n$.
So we can choose $v\in\R^n$ with $v^Ta_j\ne v^Ta_i$ for all $j\in J$. Now set
$q:=\prod_{j\in J}(v^Ta_j-v^Tx)\in\R[x]$, a polynomial of degree $\#J\le d-1$.
\end{proof}

\begin{thm}\label{relaxationhierarchy}
Let $p\in\R[x]$ be a product of linear polynomials with $p(0)\ne0$. Then the following hold:
\begin{enumerate}[(a)]
\item $C(p)\subseteq S^{(d)}(p)$ for each $d\in\N_0$\qquad(``relaxation'')
\item $S^{(0)}(p)\supseteq S^{(1)}(p)\supseteq S^{(2)}(p)\supseteq S^{(3)}(p)\supseteq\ldots$\qquad(``hierarchy'')
\item If $d:=\deg p\ge1$, then $C(p)=S^{(d-1)}(p)$\qquad(``finite convergence'')
\end{enumerate}
\end{thm}

\begin{proof}
WLOG $p(0)=1$. Write $p=\prod_{i=1}^d(1+a_i^Tx)$ with $a_1,\ldots,a_d\in\R^n$ and $d=\deg p$.
For example by Proposition \ref{rcsdet}
(interpreting the product representation of $p$ as a diagonal determinantal representation), we have
\[C(p)=\{a\in\R^n\mid1+a_1^Ta\ge0,\ldots,1+a_d^Ta\ge0\}.\]

(a) Let $a\in C(p)$. We have to show $M_p^{(e)}(a)\succeq0$ for all $e\in\N_0$. By Lemma \ref{pencilevalhierarchy}, this is equivalent to
$L_p(q^2(1+a^Tx))\ge0$ for all $q\in\R[x]$. This means by Proposition \ref{dirac} that
\[\sum_{i=1}^dq(a_i)^2(1+a^Ta_i)\ge0\]
for all $q\in\R[x]$. But even more is true: For each $i\in\{1,\dots,d\}$,
$1+a^Ta_i=1+a_i^Ta\ge0$
since $a\in C(p)$ and therefore $q(a_i)^2(1+a^Ta_i)\ge0$.

(b) is clear from Definition \ref{dfpencilhierarchy}.

(c) One inclusion has been proven already in (a). For the other one, let $a\in\R^n\setminus C(p)$. We show that $a\notin S^{(d-1)}(p)$.
By Lemma \ref{pencilevalhierarchy} and Proposition \ref{dirac}, this means we have to show that there exists a polynomial $q\in\R[x]\setminus\{0\}$ with
$\deg q\le d-1$ such that \[\sum_{j=1}^dq(a_j)^2(1+a_j^Ta)<0.\]
Choose $i\in\{1,\ldots,d\}$ such that $1+a_i^Ta<0$. By Lemma \ref{interpol}, we can choose a polynomial $q\in\R[x]$ with $\deg q<d$ such
that $q(a_i)\ne0$ and $q(a_j)=0$ for all $j\in\{1,\ldots,d\}$ with $a_j\neq a_i$. Then
\[\sum_{j=1}^dq(a_j)^2(1+a_j^Ta)=kq(a_i)^2(1+a_i^Ta)<0,\]
where $k=\#\{j\mid a_i=a_j\}$.
\end{proof}


\section{Tightening the relaxation}\label{sTighenRelax}

Let $p\in\R[x]$ be a real zero polynomial. By Theorem \ref{relaxation}, $S(p)$ is an outer spectrahedral
approximation of $C(p)$, i.e., $C(p)\subseteq S(p)$. If $p$ has high degree, then we cannot expect
in general that this is a good approximation since $S(p)$ is defined by a very small linear matrix inequality.
In this section, we analyze qualitatively a very simple idea of how to improve the spectrahedral outer approximation. The price we will have to pay is of course that we will need more linear matrix
inequalities (of the same size however).
Roughly, the idea is to ``move the origin''. More precisely, choose a point
$a\in C(p)\setminus Z(p)$ that is different from the origin. By Theorem \ref{rzshift}, the polynomial
$p(x+a)$ is again a real zero polynomial and we have $C(p(x+a))+a=C(p)$. In general, we do however
not have that $S(p(x+a))+a=S(p)$. This seems like a theoretical defect at first sight but
turns out to be a fortunate fact that we can take advantage of. Namely, we have
$C(p(x+a))\subseteq S(p(x+a))$ by Theorem \ref{relaxation} and therefore
$C(p)=C(p(x+a))+a\subseteq S(p(x+a))+a$ so that $S(p(x+a))+a$ is another outer spectrahedral
relaxation of $C(p)$ that will in general be different from $S(p)$.
Hence the intersection \[S(p)\cap(S(p(x+a))+a)\] will in general be an improved outer approximation
of $C(p)$. It is defined by two linear matrix inequalities of size $n+1$ each
which could of course be combined into a single one of size $2n+2$.
Instead of choosing two points inside $C(p)\setminus Z(p)$, namely the origin and $a$, we could now more
generally choose
finitely many points $a_1,\ldots,a_k\in C(p)\setminus Z(p)$ (the origin need not be among
them) and consider the spectrahedron
\[\bigcap_{i=1}^k(S(p(x+a_i))+a_i)\supseteq C(p)\]
defined by a linear matrix inequality of size $k(n+1)$. In practice, it seems like $S(p(x+a))+a$ tightly
approximates $C(p)$ in a neighborhood of $a\in C(p)\setminus Z(p)$. If this is right, then one would
get a very tight outer approximation by choosing $a_1,\ldots,a_k\in C(p)\setminus Z(p)$ in such way
that each point in $C(p)\cap Z(p)$ is close to one of the $a_i$. This might of course have a very
large price, namely that the number of points $k$ might have to be very large.

Now we want to support the just presented view on how to make the approximation tighter. To this end,
we prove a rather theoretical result that says that if we take \emph{all} points of $C(p)\setminus Z(p)$
instead of just finitely many, then the corresponding intersection equals $C(p)$. This is Corollary
\ref{intersectioncor} below. It does of course not imply that $C(p)$ is a spectrahedron since we deal
now with an \emph{infinite} intersection. In fact, we prove a more precise theorem that provides
evidence for our idea that it might be sufficient to choose
the points close to the boundary of $C(p)$ provided each boundary point is close to one of the $a_i$.
This is Theorem \ref{intersection} where we intersect not over all points of $C(p)\setminus Z(p)$ but
only over those lying outside of a fixed subset $D$ of $C(p)\setminus Z(p)$ which is closed in $\R^n$. One can picture
$C(p)$ as a potato, $Z(p)$ as its skin and $D$ as the peeled potato (where the removed part inevitably
is a bit more than the skin).

Our theoretical theorem actually works even for a certain polyhedral instead of spectrahedral
outer approximation that we will now introduce. This is a poor man's version of the spectrahedron introduced in Definition \ref{dfpencil}(a). When no intersection comes into play, we actually use
only an affine half-space or, in exceptional cases, the full space, namely the one
that is defined by the linear inequality that corresponds to the top left entry of the pencil that
defines the spectrahedron.

\begin{df}\label{halfspace}
Let $p\in\R[[x]]$ be a power series with $p(0)\ne0$ and $d\in\N_0$.
We call
\begin{align*}
P_d(p)&:=\{a\in\R^n\mid L_{p,d}(1)+L_{p,d}(x_1)a_1+\ldots+L_{p,d}(x_n)a_n\ge0\}
\end{align*}
the \emph{polyhedron associated to $p$ with respect to the virtual degree $d$}.
\end{df}

\begin{rem}\label{polyhedral-relaxation-remark}
\begin{enumerate}[(a)]
\item If $p\in\R[[x]]$ satisfies $\trunc_1p=1+a_1x_1+\ldots+a_nx_n$, then
\[P_d(p)=\{b\in\R^n\mid d+a_1b_1+\ldots+a_nb_n\ge0\}\]
by Definition \ref{dfriesz}, which gives $L_{p,d}(1)=d$ and $L_{p,d}(x_i)=a_i$.
\item Let $p\in\R[[x]]$ with $p(0)\ne0$. If $\trunc_1p$ is constant, then $P_d(p)=\R^n$. Otherwise $P_d(p)$ is an affine half-space.
\item If $p\in\R[x]$ is a real zero polynomial of degree at most $d$, then we have \[C(p)\subseteq S_d(p)\subseteq P_d(p)\] where the first
inclusion follows from Theorem \ref{relaxation} and the second is trivial.
\item Let $p\in\R[x]$ be a real zero polynomial of degree at most $d$. The inclusion \[C(p)\subseteq P_d(p)\]
is elementary, in contrast to the finer statement in (c). To show it, suppose WLOG $p(0)=1$, $\deg p=d$ and write $\trunc_1p=1+a_1x_1+\ldots+a_nx_n$
with $a_1,\ldots,a_n\in\R$. We have to show that $d+a_1b_1+\ldots+a_nb_n\ge0$ for all $b\in C(p)$. Fixing $b\in C(p)$,
write \[(*)\qquad p(tb)=\prod_{i=1}^e(1+d_it)\] for some $e\in\{0,\ldots,d\}$ and $d_1,\ldots,d_e\in\R^\times$. By Definition \ref{dfrcs}, we have that
$p(tb)$ has no roots in the interval $(0,1)$. In other words, we have $d_i\ge-1$ for all $i\in\{1,\ldots,e\}$. Extracting the coefficient of $t$ on both
sides of $(*)$, we get $a_1b_1+\ldots+a_nb_n=d_1+\ldots+d_e$ and therefore
\[d+a_1b_1+\ldots+a_nb_n=d+d_1+\ldots+d_e\ge d-e\ge0.\]
\end{enumerate}
\end{rem}

\begin{thm}\label{intersection}
Suppose $d\in\N_0$, $p\in\R[x]$ is a real zero polynomial of degree at most $d$, and $D\subseteq C(p)\setminus Z(p)$ is closed in $\R^n$. Then
\[C(p)=\bigcap_{a\in C(p)\setminus(Z(p)\cup D)}\bigl(S_d(p(x+a))+a\bigr)=\bigcap_{a\in C(p)\setminus(Z(p)\cup D)}\bigl(P_d(p(x+a))+a\bigr)\]
where the empty intersection is interpreted as $\R^n$.
\end{thm}

\begin{proof}
Recall that $p(x+a)$ is a real zero polynomial for each $a\in C(p)\setminus Z(p)$ by Theorem \ref{rzshift}.

Both inclusions from left to right follow from Theorem \ref{relaxation} and Remark \ref{polyhedral-relaxation-remark}(c).
For each $a\in C(p)\setminus Z(p)$, we have $C(p)=C(p(x+a))+a$. Since $p(x+a)$ has degree at most $d$, Theorem \ref{relaxation} gives $C(p(x+a))\subseteq S_d(p(x+a))$, and Remark \ref{polyhedral-relaxation-remark}(c) gives $S_d(p(x+a))\subseteq P_d(p(x+a))$. Therefore
$C(p)\subseteq S_d(p(x+a))+a\subseteq P_d(p(x+a))+a$.

It remains to show that
\[C(p)\supseteq\bigcap_{a\in C(p)\setminus(Z(p)\cup D)}P_d(p(x+a))+a.\]
We even show that for each half-line $H$ emanating from the origin,
\[H\cap C(p)\supseteq H\cap\bigcap_{a\in H\cap(C(p)\setminus(Z(p)\cup D))}(P_d(p(x+a))+a),\]
where the empty intersection stands of course again for $\R^n$.
Using Proposition \ref{applyMatrixToLFrom}, one easily reduces to the case where $H$ is the positive first axis
\[H=\R_{\ge0}\times\{0\}\subseteq\R^n.\]
By Lemma \ref{restriction}(a) and Proposition \ref{applyMatrixToLFrom}, we can reduce to the case $n=1$.
The univariate polynomial $p(x_1,0,\ldots,0)$ again has degree at most $d$. 
So suppose from now on that
we have just one variable $x=x_1$. Then $H=\R_{\ge0}$. WLOG $p(0)=1$. Write
\[p=\prod_{i=1}^d(1+a_ix)\]
with $a_1,\ldots,a_d\in\R$. The roots of $p$ are then the $-\frac1{a_i}$ with $a_i\ne0$.
If none of these roots is positive, then $C(p)\cap H=\R_{\ge0}$ and there is nothing to show.
From now on we can therefore
suppose that $d\ge1$ and \[r:=-\frac1{a_1}\] is the smallest positive root of $p$ so that
$H\cap C(p)=[0,r]$. In particular, $r>0$ and $a_1<0$. Since $D$ is closed in $\R^n$ and $r\notin D$, we can choose $\ep>0$ such that $r-\ep>0$ and
$H\cap D\subseteq[0,r-\ep]$. Then $(r-\ep,r)\subseteq H\cap(C(p)\setminus(Z(p)\cup D))$.
It therefore is enough to show that
\[(*)\qquad \bigcap_{a\in(r-\ep,r)}P_d(p(x+a))+a\subseteq(-\infty,r].\]
For $a\in\R\setminus Z(p)\supseteq(r-\ep,r)$, the polynomial
\[(**)\qquad p_a:=\frac{p(x+a)}{p(a)}=\prod_{i=1}^d\frac{1+a_i(x+a)}{1+a_ia}=\prod_{i=1}^d\left(1+\frac{a_i}{1+a_ia}x\right)\]
has constant coefficient $1$. Proposition \ref{dirac} hence implies that
$L_{p_a,d}$ is for each $a\in\R\setminus Z(p)\supseteq(r-\ep,r)$ integration with respect to the
sum of the Dirac measures in the points $\frac{a_1}{1+a_1a},\ldots,\frac{a_d}{1+a_da}$.
Now suppose that $b$ lies in the left hand side of $(*)$.
Then we have
\[d+(b-a)\sum_{i=1}^d\frac{a_i}{1+a_ia}=L_{p_a,d}(1)+(b-a)L_{p_a,d}(x)\ge0\]
for all $a\in(r-\ep,r)$. Now let $a$ converge to $r$ from below.
We have that $1+a_1a$ converges to $0$ from above.
Hence the summand $\frac{a_1}{1+a_1a}$ converges to $-\infty$. The summand $\frac{a_i}{1+a_ia}$ shows the same
behavior in the case where $a_i=a_1$. All other summands converge to some real number.
Hence the whole sum converges to $-\infty$. The factor $b-a$ converges to $b-r$ from above. If $b-r$ were positive, then
\[d+(b-a)\sum_{i=1}^d\frac{a_i}{1+a_ia}\]
would converge to $-\infty$ while remaining nonnegative.
Hence $b-r\le0$, so $b$ lies in $(-\infty,r]$, the right-hand side of $(*)$.
\end{proof}

\begin{cor}\label{intersectioncor}
Let $p\in\R[x]$ be a real zero polynomial.
Then
\[C(p)=\bigcap_{a\in C(p)\setminus Z(p)}\bigl(S(p(x+a))+a\bigr).\]
\end{cor}


\section{Exactness for quadratic real zero polynomials}\label{sQuadraticExact}

It is quite trivial that $C(p)=S(p)$ for linear real zero polynomials $p\in\R[x]$: if $p=1+a^Tx$ with $a\in\R^n$, then
$M_p(b)=(1+a^Tb)\binom1a\binom1a^T$ for all $b\in\R^n$ by Proposition \ref{dirac}, so that $S(p)=\{b\in\R^n\mid1+a^Tb\ge0\}=C(p)$.
Since $S(p)$ depends only on the cubic part
$\trunc_3p$ of $p$ by Remark \ref{dependsonlyoncubicpart}, there seems to be no way
that $C(p)=S(p)$ in general if the degree of the real zero polynomial $p$ is bigger than three.
But worse than that, one lacks in general exactness also for cubic real zero polynomials as the reader
will easily find. The next theorem shows that our relaxation at least is exact for quadratic real zero polynomials. 

\begin{thm}\label{quadratic}
Let $p$ be a quadratic real zero polynomial. Then $C(p)=S(p)$.
\end{thm}

\begin{proof}
From Theorem \ref{relaxation}, we know already $C(p)\subseteq S(p)$. To show the reverse $S(p)\subseteq C(p)$, we show that it holds
after intersecting with an arbitrary line through the origin. Instead of an arbitrary line, we can without loss of generality consider the
first axis in $\R^n$ by Proposition \ref{applyMatrixToLFrom}.
By Lemma \ref{restriction}(c), it is enough to show $S_{\deg p}(q)\subseteq C(q)$ for
$q:=p(x_1,0,\ldots,0)\in\R[x_1]$. If $\deg q=0$, then $C(q)=\R$ and there is nothing to show.
If $\deg q=1$, say $q=1+ax_1$ with a non-zero $a\in\R$, then the bottom right entry of $M_{q,\deg p}$ is $a^2+x_1a^3$ by Proposition
\ref{dirac} and hence $S_{\deg p}(q)\subseteq\{b\in\R\mid a^2+a^3b\ge0\}=\{b\in\R\mid 1+ab\ge0\}$. Finally, if
$\deg q=2$, then $S_{\deg p}(q)=S(q)=S^{(1)}(q)=C(q)$ by Theorem \ref{relaxationhierarchy}.
\end{proof}

The preceding proof makes use of the Helton--Vinnikov Theorem \ref{vinnikov} or Corollary \ref{vcor} indirectly through Theorem \ref{relaxation}. 
However, we need it only for polynomials of degree at most two in two variables. So we can use Example \ref{hv2} instead (together with the trivial cases of degree at most one) and therefore our proof is self-contained.

The next corollary was explicitly mentioned by Netzer and Thom \cite[Corollary 5.4]{nt} but most likely
was known before. Netzer and Thom use a hermitian linear matrix inequality of size
$2^{\lfloor\frac n2\rfloor}$ to describe a rigidly convex set given by a quadratic polynomial in $\R^n$ whereas
we need just a symmetric linear matrix inequality of size only $n+1$. While the proof of Netzer and Thom
is still very interesting for other reasons,
the result could actually also be proven by reducing it, via projective space, to the case of the unit ball
(cf. \cite[Page 958]{gar}). The well-known description of the unit ball in $\R^n$ by a
symmetric linear matrix inequality of size $n$ (see for example \cite[Page 591]{kum1}) shows that one can
for $n\ge2$ even get down to size $n$ instead of $n+1$. If $n=2^k+1$ for some $k\in\N_0$, then Kummer
shows that $n$ is the minimal size of a real symmetric linear matrix inequality describing the unit ball.
In general, he shows that $\frac n2$ is a lower bound \cite[Theorem 1]{kum1}.

\begin{cor}\label{quspec}
The rigidly convex set defined by a quadratic real zero polynomial is always a spectrahedron.
\end{cor}

\section{Tightening rigidly convex sets}\label{sTightening}

In this section, let $d\in\N$ and $y=(y_1,\ldots,y_n)$ and $z=(z_1,\ldots,z_{d-1})$. For $p\in\R[x]$ with $p(0)\neq0$, we extend the definition of $L_p$ to a linear map $\R[x,y,z]\to\R[y,z]$ so that the $y_i$ and $z_j$ are treated as constants, that is,
\begin{equation*}
    L_p(z^\gamma y^\beta x^\alpha)=z^\gamma y^\beta L_p(x^\alpha).
\end{equation*}

\begin{thm}\label{RigidlyConvexSetAsSliceOfNonNegPolys}
    Let $p\in\R[x]$ be a real zero polynomial of degree $d$. Then the rigidly convex set $C(p)$ of $p$ is equal to
    \begin{equation*}
        \{a\in\R^n\mid L_p((1+a^Tx)(y^Tx-z_1)^2\cdots(y^Tx-z_{d-1})^2)\geq0\text{ on }\R^n\times\R^{d-1}\}.
    \end{equation*}
\end{thm}
\begin{proof}
    We may assume $p(0)=1$. Indeed, for $q:=\frac p{p(0)}$ we have $C(p)=C(q)$ and $L_p=L_q$.
    For the inclusion `$\subseteq$', let $a\in C(p)$, $y\in\R^n$ and $z\in\R^{d-1}$. By Corollary \ref{vcor}, after padding by zero blocks if $\deg p(sa+ty)<d$ as in the proof of Theorem \ref{relaxation}, there are hermitian $A,B\in\C^{d\times d}$ with $p(sa+ty)=\det(I_d+sA+tB)$. Then
    \begin{align*}
        &L_p((1+a^Tx)(y^Tx-z_1)^2\cdots(y^Tx-z_{d-1})^2)\\
        &\stackrel{\ref{LFormOfRidgePolyTimesLinearPoly}}{=}\tr((I_d+A)(B-z_1I)^2\cdots(B-z_{d-1}I)^2)\\
        &\geq0
    \end{align*}
    since both $I_d+A$ and $(B-z_1I)^2\cdots(B-z_{d-1}I)^2$ are positive semidefinite, the first follows from $a\in C(p)$.

    Conversely, let $a\in\R^n$ satisfy
    \begin{equation*}
        L_p((1+a^Tx)(y^Tx-z_1)^2\cdots(y^Tx-z_{d-1})^2)\geq0\text{ on }\R^n\times\R^{d-1}.
    \end{equation*}
    Suppose $a\notin C(p)$. Then there is $\lambda\in(0,1)$ with $p(\lambda a)=0$. Since $p$ is real zero and $p(0)=1$, we can write
    \[p(ta)=\prod_{i=1}^d(1+\zeta_it)\]
    with $\zeta_i\in\R$, where factors $\zeta_i=0$ account for $\deg(p(ta))<d$, and where $\zeta_d=-\lambda^{-1}<-1$. Rearrange the roots so that $\zeta_{k+1}=\cdots=\zeta_d$ and $\zeta_d\notin\{\zeta_1,\ldots,\zeta_k\}$ for some $k\in\{0,\ldots,d-1\}$. Set $y:=a$, and let the coordinates of $z\in\R^{d-1}$ be $\zeta_1,\ldots,\zeta_k$, followed by $d-k-1$ zeros. Empty products are read as $1$, and a power with exponent $0$ as $1$. With
    \[q:=(1+t)\prod_{j=1}^k(t-\zeta_j)^2\,t^{2(d-k-1)}\in\R[t]\]
    and with $C$ the companion matrix of $p$, Proposition \ref{RestrictLFormToLine} yields
    \[L_p(q(a^Tx))=\tr(q(-C(a))).\]
    The univariate polynomial $r:=p(ta)\in\R[t]$ equals $\det(I_d-tC(a))$, so the eigenvalues of $C(a)$ are $-\zeta_1,\ldots,-\zeta_d$. Hence $\tr(q(-C(a)))=\sum_{i=1}^d q(\zeta_i)$. Proposition \ref{dirac}, applied to $r$, gives the same value. For $i\le k$ the summand $q(\zeta_i)$ vanishes, and $\zeta_{k+1}=\cdots=\zeta_d$, so
    \[\sum_{i=1}^d q(\zeta_i)=(d-k)(1+\zeta_d)\prod_{j=1}^k(\zeta_d-\zeta_j)^2\,\zeta_d^{2(d-k-1)}.\]
    Here $d-k\ge1$ and $1+\zeta_d<0$. The product is positive because $\zeta_d$ is distinct from $\zeta_1,\ldots,\zeta_k$, and $\zeta_d^{2(d-k-1)}>0$. The sum is therefore negative, contradicting the hypothesis. Thus $a\in C(p)$.
\end{proof}

\begin{rem}
    Nonnegativity on $\R^n\times\R_{\geq-1}^{d-1}$ already suffices. If $a\notin C(p)$, write $p(ta)=\prod_{i=1}^d(1+\zeta_it)$ as in the proof and let the coordinates of $z$ be the roots $\zeta_i>-1$, filled up by zeros if necessary. These coordinates lie in $\R_{\geq-1}$. For the corresponding univariate polynomial $q$, one has $q(\zeta_i)=0$ whenever $\zeta_i>-1$, whereas $q(\zeta_i)<0$ whenever $\zeta_i<-1$. The evaluation in the proof is therefore negative. Hence the rigidly convex set of $p$ is
    \begin{equation*}
        \{a\in\R^n\mid L_p((1+a^Tx)(y^Tx-z_1)^2\cdots(y^Tx-z_{d-1})^2)\geq0\text{ on }\R^n\times\R_{\geq-1}^{d-1}\}.
    \end{equation*}
\end{rem}

\begin{rem}
    Like in \cite[Section 3]{kpv}, we can strengthen the condition of the homogeneous polynomial
    \begin{equation*}
        L_p((1+a^Tx)(y^Tx-z_1)^2\cdots(y^Tx-z_{d-1})^2)
    \end{equation*}
    being non-negative on $\R^n\times\R^{d-1}$ to get a tightening of the rigidly convex set of $p$, that is, we get a set contained in $C(p)$. See Subsection \ref{sectionKPV} for the connection to \cite{kpv}.
\end{rem}

\begin{df}\label{dfTightening}
    Let $p\in\R[x]$ be a real zero polynomial of degree $d$ and let $k\in\N_0$. Then define the $k$-\emph{tightening} of the rigidly convex set of $p$ by
    \begin{multline*}
        R_k(p):=\{a\in\R^n\mid\|y\|^{2k}L_p((1+a^Tx)(y^Tx-z_1)^2\cdots(y^Tx-z_{d-1})^2)\\
        \text{ is a sum of squares in }\R[y,z]\},
    \end{multline*}
    where $\|y\|^2=y_1^2+\ldots+y_n^2$.
\end{df}

\begin{rem}\label{remMultiplier}
    One could also use the multiplier $(\|y\|^2+\|z\|^2)^k$ instead of $\|y\|^{2k}$ in Definition \ref{dfTightening}; Corollary \ref{TighteningIsATightening} below would remain true with the same proof. We have chosen $\|y\|^{2k}$ because of Theorem \ref{StrictTightening} below: the form $L_p((1+a^Tx)(y^Tx-z_1)^2\cdots(y^Tx-z_{d-1})^2)$ is never positive definite on $\R^n\times\R^{d-1}$ if $d\ge3$, since it vanishes at all points with $y=0$ and $z_1\cdots z_{d-1}=0$ (see Remark \ref{remZerosAtYZero}), so that Reznick's theorem on uniform denominators cannot be applied directly. For each fixed $y\ne0$ it is strictly positive for every $z$. Lemma \ref{HermitePD} gives the matrix form of this fact: the Hermite matrix $H_a(y)$ is positive definite. This is what the multiplier $\|y\|^{2k}$ exploits.
\end{rem}

\begin{df}[{\cite[Section 3.1]{np}}]
    A \emph{spectrahedral shadow} is the image of a spectrahedron under a linear map.
\end{df}

\begin{ex}[{\cite[Example 3.1.2]{np}}]\label{SosSpectrahedralShadow}
    The sums of squares in $\R[x]_{2d}$ form a spectrahedral shadow. It is the image of the cone of positive semidefinite matrices under a Gram map
    \begin{equation*}
        \Sym_N(\R)\to\R[x]_{2d},\, A\mapsto m^TAm,
    \end{equation*}
    where $m$ is a vector of all monomials of degree $d$ and $\Sym_N(\R)$ is the set of all symmetric matrices of size $N$ (which is the length of $m$).
\end{ex}

\begin{cor}\label{TighteningIsATightening}
    Let $p\in\R[x]$ be a real zero polynomial of degree $d$ and $k\in\N_0$. Then
    \begin{equation*}
        R_0(p)\subseteq R_k(p)\subseteq R_{k+1}(p)\subseteq C(p).
    \end{equation*}
    Furthermore, $R_k(p)$ is a spectrahedral shadow.
\end{cor}
\begin{proof}
    The inclusions $R_0(p)\subseteq R_k(p)\subseteq R_{k+1}(p)$ hold since the product of $\|y\|^2$ with a sum of squares is again a sum of squares. If $a\in R_k(p)$, then the polynomial $L_p((1+a^Tx)(y^Tx-z_1)^2\cdots(y^Tx-z_{d-1})^2)$ is nonnegative on $\{(y,z)\mid y\ne0\}$ and hence, by continuity, on all of $\R^n\times\R^{d-1}$, so that $a\in C(p)$ by Theorem \ref{RigidlyConvexSetAsSliceOfNonNegPolys}. For the last statement, note that $a\mapsto\|y\|^{2k}L_p((1+a^Tx)(y^Tx-z_1)^2\cdots(y^Tx-z_{d-1})^2)$ is an affine linear map from $\R^n$ to a finite-dimensional space of polynomials, and that $R_k(p)$ is the preimage of the cone of sums of squares under this map. Preimages of spectrahedral shadows under affine linear maps are spectrahedral shadows, so the claim follows from Example \ref{SosSpectrahedralShadow}.
\end{proof}

\begin{rem}
    If one has a real zero polynomial $p$ and wants to show that its rigidly convex set is a spectrahedral shadow, one approach is to show that $R_k(p)=C(p)$ for some $k\in\N_0$. Of course, this is not always possible as the counterexample \cite{OpenAI2} by OpenAI showed.
\end{rem}

Let us explore some cases where $R_0(p)=C(p)$. The first one is trivial but worth noting.

\begin{pro}\label{QuadraticTightening}
    Let $p\in\R[x]$ be a real zero polynomial of degree $d\le2$. Then $R_0(p)=C(p)$.
\end{pro}
\begin{proof}
    For $d=1$, the polynomial $L_p((1+a^Tx))$ is a constant, which is nonnegative for $a\in C(p)$ by Theorem \ref{RigidlyConvexSetAsSliceOfNonNegPolys} and hence a sum of squares. For $d=2$, the polynomial $L_p((1+a^Tx)(y^Tx-z_1)^2)$ is a quadratic form in the $n+1$ variables $(y,z_1)$. For $a\in C(p)$ it is nonnegative on $\R^n\times\R$ by Theorem \ref{RigidlyConvexSetAsSliceOfNonNegPolys}, i.e., a positive semidefinite quadratic form, and hence a sum of squares of linear forms.
\end{proof}

\begin{thm}\label{DeterminantalTightening}
    Let $p\in\R[x]$ be a real zero polynomial of degree $d$ and let $k\in\N$. Suppose that there are hermitian matrices $A_1,\ldots,A_n\in\C^{dk\times dk}$ with
    \[p^k=\det(I_{dk}+x_1A_1+\ldots+x_nA_n).\]
    Then $R_0(p)=C(p)$. In particular, $R_0(p)=C(p)$ if $p$ itself has a determinantal representation $p=\det(I_d+x_1A_1+\ldots+x_nA_n)$ with hermitian $A_i\in\C^{d\times d}$.
\end{thm}
\begin{proof}
    One inclusion is just Corollary \ref{TighteningIsATightening}. For the other inclusion, let $a\in C(p)=C(p^k)$. Write $A:=x_1A_1+\ldots+x_nA_n\in\C[x]^{dk\times dk}$. By Proposition \ref{rcsdet}, $a\in C(p^k)$ means that $I+A(a)$ is positive semidefinite. Let $B\in\C^{dk\times dk}$ be the positive semidefinite square root of $I+A(a)$. Since $p^k(0)=1$, Proposition \ref{logprop}(a) gives $L_{p^k,dk}=kL_{p,d}=kL_p$. For fixed $b\in\R^n$ and $c\in\R^{d-1}$, Proposition \ref{LFormOfRidgePolyTimesLinearPoly}, applied to $p^k$ with virtual degree $dk$ and $q:=(t-c_1)^2\cdots(t-c_{d-1})^2$, yields
    \begin{multline*}kL_p((1+a^Tx)(b^Tx-c_1)^2\cdots(b^Tx-c_{d-1})^2)\\=\tr\bigl((I+A(a))(A(b)-c_1I)^2\cdots(A(b)-c_{d-1}I)^2\bigr).\end{multline*}
    Both sides are polynomials in $(b,c)$, so this is an identity in $\R[y,z]$ after replacing $(b,c)$ by $(y,z)$. Set $M:=(A(y)-z_1I)\cdots(A(y)-z_{d-1}I)\in\C[y,z]^{dk\times dk}$. The factors commute and are hermitian for real arguments, so $M(b,c)$ is hermitian for all $(b,c)\in\R^n\times\R^{d-1}$. Using cyclic invariance of the trace, we get
    \begin{multline*}kL_p((1+a^Tx)(y^Tx-z_1)^2\cdots(y^Tx-z_{d-1})^2)=\tr((I+A(a))M^2)\\=\tr(BBMM)=\tr(MBBM)=\tr(N^*N),\end{multline*}
    where $N:=BM\in\C[y,z]^{dk\times dk}$ and $N^*$ denotes the matrix obtained from $N^T$ by conjugating all coefficients, so that $N^*(b,c)=N(b,c)^*$ for real $(b,c)$. Writing $N_{ij}=u_{ij}+\ii v_{ij}$ with $u_{ij},v_{ij}\in\R[y,z]$, we obtain
    \[kL_p((1+a^Tx)(y^Tx-z_1)^2\cdots(y^Tx-z_{d-1})^2)=\sum_{i,j=1}^{dk}\bigl(u_{ij}^2+v_{ij}^2\bigr).\]
    Dividing by $k$ shows that $L_p((1+a^Tx)(y^Tx-z_1)^2\cdots(y^Tx-z_{d-1})^2)$ is a sum of squares in $\R[y,z]$. Hence $a\in R_0(p)$.
\end{proof}

\begin{cor}\label{DetCor}
    Let $p\in\R[x]$ be a real zero polynomial of degree $d$. Then $R_0(p)=C(p)$ if one of the following holds:
    \begin{enumerate}[(a)]
        \item $n=2$.
        \item $n=3$ and $d=3$.
    \end{enumerate}
\end{cor}
\begin{proof}
    We may assume $p(0)=1$. Item (a) follows from Corollary \ref{vcor} and Theorem \ref{DeterminantalTightening}.

    For Item (b), by Theorem \ref{DeterminantalTightening} it suffices to find hermitian $A_1,A_2,A_3\in\C^{3\times3}$ such that $p=\det(I_3+x_1A_1+x_2A_2+x_3A_3)$. Let $P\in\R[x_0,x_1,x_2,x_3]$ be the homogenization of $p$; it is a cubic form which is hyperbolic with respect to $e:=(1,0,0,0)$ by Proposition \ref{hbrz}. By \cite[Corollary 11.9]{hk}, $P$ has a definite hermitian determinantal representation: there are hermitian $U_0,\ldots,U_3\in\C^{3\times3}$ with $P=\det(x_0U_0+\ldots+x_3U_3)$ and $U(e)$ definite. Since $\det(U(e))=P(e)=1>0$ and the matrices have odd size, $U(e)$ is positive definite. Setting $A_i:=U_0^{-1/2}U_iU_0^{-1/2}$ and using $P(e)=1$ gives $p=\det(I_3+x_1A_1+x_2A_2+x_3A_3)$. The smooth case is due to Buckley and Ko{\v{s}}ir \cite[Theorem 6.4]{bk}. The singular case is included in the cited corollary; its proof uses Nuij's approximation by strictly hyperbolic polynomials \cite{nui}.
\end{proof}

\begin{rem}\label{remZerosAtYZero}
    Let $p\in\R[x]$ be a real zero polynomial of degree $d$ and $a\in\R^n$. Substituting $y=0$ in $(1+a^Tx)(y^Tx-z_1)^2\cdots(y^Tx-z_{d-1})^2$ gives $(1+a^Tx)z_1^2\cdots z_{d-1}^2$, so that
    \[L_p((1+a^Tx)(y^Tx-z_1)^2\cdots(y^Tx-z_{d-1})^2)\big|_{y=0}=z_1^2\cdots z_{d-1}^2\,L_p(1+a^Tx).\]
    Hence this polynomial vanishes at every point $(0,c)$ with $c_1\cdots c_{d-1}=0$. For $d\ge3$ these zeros form a nontrivial union of linear subspaces, so the polynomial is not a positive definite form on $\R^n\times\R^{d-1}$, cf.\ Remark \ref{remMultiplier}.
\end{rem}

\begin{df}\label{dfStrict}
    A polynomial $p\in\R[x]$ of degree $d$ is called \emph{strict real zero} if $p(0)\ne0$ and for every $a\in\R^n\setminus\{0\}$ the univariate polynomial
    \begin{equation*}
        p(ta)\in\R[t]
    \end{equation*}
    has exactly $d$ distinct roots in $\R\cup\{\infty\}$. Here $\infty$ counts as a root if $\deg(p(ta))<d$.

    So, $p$ is strict real zero if and only if for any $a\in\R^n\setminus\{0\}$ there are pairwise distinct $z_1,\ldots,z_d\in\R$ such that
    \begin{equation*}
        p(ta)=p(0)(1+z_1t)\cdots(1+z_dt).
    \end{equation*}
    In particular, every strict real zero polynomial is a real zero polynomial. The dehomogenization (Proposition \ref{hbrz}) of a strictly hyperbolic polynomial in the sense of \cite{nui,KummerNetzer} is strict real zero. We will show in Theorem \ref{StrictTightening} below that the interior of $C(p)$ is exhausted by the tightenings $R_k(p)$ if $p$ is strict real zero. For this we need the Hermite matrix, which we introduce next.
\end{df}

\subsection{Connection to Hermite and Bézout matrices}\label{ssHermiteBezout}
There is a classical result originally due to Hermite on how to count the number of real and complex roots of a monic univariate polynomial. In the following, we will introduce a multivariate analog for real zero polynomials. This can be found, for example, in \cite{npt}. Another method is to use the Bézout matrix of a polynomial and its derivative. We will also use a generalization thereof. This connection of real zero polynomials and the Bézout matrix has been exploited before; see \cite{kum2}.

\begin{df}\label{dfHermiteMatrix}
    Let $p\in\R[x]$ be a polynomial of degree $d$ with $p(0)\neq0$ and denote by $C\in\R[x]^{d\times d}$ its companion matrix. Then the \emph{Hermite matrix} of $p$ is the matrix
    \begin{equation*}
        H=(\tr(C^{i+j-2}))_{i,j=1,\ldots,d}\in\R[x]^{d\times d}.
    \end{equation*}
\end{df}

\begin{df}
    We call a symmetric matrix polynomial $M\in\R[x]^{d\times d}$ \emph{positive semidefinite} if $M(a)$ is positive semidefinite for all $a\in\R^n$.
\end{df}

\begin{pro}[{\cite[Corollary 1.2]{npt}}]
    Let $p\in\R[x]$ be a polynomial with $p(0)\neq0$. Then $p$ is real zero if and only if its Hermite matrix is positive semidefinite.
\end{pro}
\begin{proof}
    This is an easy consequence of Hermite's method for counting roots of univariate polynomials; see, for example, \cite[Theorem 4.59]{bpr}.
\end{proof}

\begin{pro}\label{HermiteForm}
    Let $p\in\R[x]$ be a polynomial of degree $d$ with $p(0)\ne0$. Denote by $C\in\R[x]^{d\times d}$ its companion matrix and by $H\in\R[x]^{d\times d}$ its Hermite matrix. Then 
    \begin{equation*}
        H(y)=(L_p((-y^Tx)^{i+j-2}))_{i,j=1,\ldots,d}.
    \end{equation*}
    More generally, for $a\in\R^n$ set
    \begin{equation*}
        H_a:=\left(L_p\left((1+a^Tx)(-y^Tx)^{i+j-2}\right)\right)_{i,j=1,\ldots,d}\in\R[y]^{d\times d},
    \end{equation*}
    so that $H_0=H(y)$ arises from $H$ by substituting $y$ for $x$. Then the $(i,j)$-entry of $H_a$ is a form of degree $i+j-2$ in $y$, and
    \begin{equation*}
        L_p((1+a^Tx)(y^Tx-z_1)^2\cdots(y^Tx-z_{d-1})^2)=v^TH_av,
    \end{equation*}
    where $v=(s_{d-1},\ldots,s_1,s_0)^T$ and $s_i\in\R[z]_i$ is the $i$th elementary symmetric polynomial in $z_1,\ldots,z_{d-1}$.
\end{pro}
\begin{proof}
    The first equality follows immediately from Proposition \ref{RestrictLFormToLine}. The statement about the degrees is clear since $L_p$ maps forms of degree $k$ in $x$ to constants. For the last equality, note that
    \[\prod_{j=1}^{d-1}(y^Tx-z_j)=\sum_{i=0}^{d-1}(-1)^is_i\,(y^Tx)^{d-1-i}=(-1)^{d-1}\sum_{i=0}^{d-1}s_i\,(-y^Tx)^{d-1-i},\]
    square both sides, multiply by $1+a^Tx$ and apply $L_p$.
\end{proof}

\begin{pro}\label{HermiteSOSImpliesTighteningExact}
    Let $p\in\R[x]$ be a real zero polynomial of degree $d$. If, for all $a\in C(p)$, the matrix $H_a\in\R[y]^{d\times d}$ from Proposition \ref{HermiteForm}
    admits a decomposition $H_a=Q_a^TQ_a$ for some $Q_a\in\R[y]^{k_a\times d}$, then $R_0(p)=C(p)$.
\end{pro}
\begin{proof}
    One inclusion is clear since $R_0(p)$ is a tightening. For the other inclusion, let $a\in C(p)$. Denote by $v$ the vector $(s_{d-1},\ldots,s_1,s_0)^T$, where $s_i\in\R[z]_i$ is the $i$th elementary symmetric polynomial. By Proposition \ref{HermiteForm},
    \begin{multline*}
        L_p((1+a^Tx)(y^Tx-z_1)^2\cdots(y^Tx-z_{d-1})^2)=v^TH_av\\
        =(Q_av)^TQ_av\in\sum\R[y,z]^2,
    \end{multline*}
    as desired.
\end{proof}

We now use the Hermite matrix to prove that the tightenings exhaust the interior of $C(p)$ for strict real zero polynomials (Definition \ref{dfStrict}).

\begin{lem}\label{HermitePD}
    Let $p\in\R[x]$ be a strict real zero polynomial of degree $d$ and let $a\in C(p)\setminus Z(p)$. Then the matrix $H_a(b)\in\R^{d\times d}$ from Proposition \ref{HermiteForm} is positive definite for every $b\in\R^n\setminus\{0\}$.
\end{lem}
\begin{proof}
    We may assume $p(0)=1$. Fix $b\ne0$ and choose hermitian $A,B\in\C^{d\times d}$ with $p(sa+tb)=\det(I_d+sA+tB)$ as in Proposition \ref{LFormOfRidgePolyTimesLinearPoly}. By that proposition, $(H_a(b))_{ij}=\tr((I_d+A)(-B)^{i+j-2})$. Every eigenvalue of $A$ is bigger than $-1$: an eigenvalue $\mu\le-1$ would give $p(\la a)=\det(I_d+\la A)=0$ for $\la=-\frac1\mu\in(0,1]$, contradicting $a\in C(p)\setminus Z(p)$. Hence $I_d+A\succ0$. Let $u\in\R^d$ and set $g:=\sum_{i=1}^du_i(-t)^{i-1}\in\R[t]$, a polynomial of degree at most $d-1$. Then
    \[u^TH_a(b)u=\tr\bigl((I_d+A)g(B)^2\bigr)=\tr\bigl((g(B)(I_d+A)^{1/2})^*(g(B)(I_d+A)^{1/2})\bigr)\ge0,\]
    with equality if and only if $g(B)=0$. Now $p(tb)=\det(I_d+tB)=\prod_{l=1}^d(1+\be_lt)$, where $\be_1,\ldots,\be_d$ are the eigenvalues of $B$. Since $p$ is strict real zero and $b\ne0$, the $\be_l$ are pairwise distinct (a root at $\infty$ of $p(tb)$ of multiplicity $m$ corresponds to exactly $m$ vanishing $\be_l$). The hermitian matrix $B$ is diagonalizable, so $g(B)=0$ if and only if $g(\be_l)=0$ for $l=1,\ldots,d$, which forces $g=0$ because $\deg g<d$, i.e., $u=0$.
\end{proof}

Reznick's theorem on uniform denominators \cite{rez} has been extended to homogeneous matrix polynomials, see \cite[Corollary 2]{dhl} and \cite[Theorem 2.3]{zf}; the proofs there, like ours, rest on the matrix version of Putinar's Positivstellensatz due to Scherer and Hol \cite[Theorem 2]{sh}, see also \cite[Theorem 3.7]{ks}. These results require all entries of the matrix polynomial to be forms of the same degree, which is not the case for the matrix $H_a$ from Proposition \ref{HermiteForm}: its $(i,j)$-entry is a form of degree $i+j-2$, and $H_a$ is singular at $y=0$. We therefore prove the variant we need. Recall that $S\in\R[y]^{d\times d}$ is called a \emph{sum of squares} (of matrix polynomials) if $S=Q^TQ$ for some $Q\in\R[y]^{m\times d}$.

\begin{lem}\label{MatrixReznick}
    Let $G\in\R[y]^{d\times d}$ be a symmetric matrix polynomial such that the $(i,j)$-entry of $G$ is a form of degree $i+j-2$ for all $i,j\in\{1,\ldots,d\}$, and suppose that $G(b)$ is positive definite for all $b\in\R^n\setminus\{0\}$. Then there is $k\in\N_0$ such that $\|y\|^{2k}G$ is a sum of squares.
\end{lem}
\begin{proof}
    For $\rh\in\R$ let $D_\rh:=\operatorname{diag}(1,\rh,\ldots,\rh^{d-1})\in\R^{d\times d}$. The degree hypothesis says precisely that
    \[(\dagger)\qquad G(\rh c)=D_\rh G(c)D_\rh\qquad(\rh\in\R,\ c\in\R^n).\]
    The quadratic module generated by $g_1:=1-\|y\|^2$ and $g_2:=-g_1$ in $\R[y]$ is Archimedean, and $G$ is positive definite on the unit sphere $\{g_1=0\}$. By the matrix version of Putinar's theorem \cite[Theorem 2]{sh}, there are sums of squares $S_0,S_1,S_2\in\R[y]^{d\times d}$ with $G=S_0+g_1S_1+g_2S_2$, i.e.,
    \[G=S_0+(1-\|y\|^2)T\]
    with $T:=S_1-S_2\in\R[y]^{d\times d}$. Write $S_0=\sum_{i=1}^rQ_i^TQ_i$ with $Q_i\in\R[y]^{m_i\times d}$ and choose $k\in\N_0$ with $k\ge d-1$ and $k\ge\deg Q_i$ (the maximal degree of an entry) for all $i$. Let $w$ be an additional variable and set $R_i:=w^kQ_i\left(\frac yw\right)D_w\in\R[y,w]^{m_i\times d}$. Decompose $R_i=E_i(y,w^2)+wO_i(y,w^2)$ into its even and odd parts with respect to $w$, where $E_i,O_i\in\R[y,u]^{m_i\times d}$ for one further variable $u$.

    Now fix $b\in\R^n\setminus\{0\}$ and let $\rh\in\{\|b\|,-\|b\|\}$. Since $\frac b\rh$ lies on the unit sphere, we have $G\left(\frac b\rh\right)=S_0\left(\frac b\rh\right)$, and $(\dagger)$ gives
    \begin{multline*}\rh^{2k}G(b)=\rh^{2k}D_\rh G\left(\frac b\rh\right)D_\rh\\=\sum_{i=1}^r\left(\rh^kQ_i\left(\frac b\rh\right)D_\rh\right)^T\left(\rh^kQ_i\left(\frac b\rh\right)D_\rh\right)=\sum_{i=1}^rR_i(b,\rh)^TR_i(b,\rh).\end{multline*}
    Adding the two identities for $\rh=\|b\|$ and $\rh=-\|b\|$ and dividing by $2$, the mixed terms cancel and we obtain
    \[\|b\|^{2k}G(b)=\sum_{i=1}^r\Bigl(E_i(b,\|b\|^2)^TE_i(b,\|b\|^2)+\|b\|^2\,O_i(b,\|b\|^2)^TO_i(b,\|b\|^2)\Bigr).\]
    Both sides are polynomial in $b$, and the identity holds on the dense set $\R^n\setminus\{0\}$, hence identically. Since $\|y\|^2O^TO=\sum_{j=1}^n(y_jO)^T(y_jO)$, the right hand side is a sum of squares, as desired.
\end{proof}

\begin{thm}\label{StrictTightening}
    Let $p\in\R[x]$ be a strict real zero polynomial. Then
    \[C(p)\setminus Z(p)\subseteq\bigcup_{k\in\N_0}R_k(p),\]
    i.e., every interior point of $C(p)$ lies in some tightening $R_k(p)$.
\end{thm}
\begin{proof}
    Let $d:=\deg p$ and $a\in C(p)\setminus Z(p)$. By Lemma \ref{HermitePD} and Proposition \ref{HermiteForm}, the matrix $G:=H_a$ satisfies the hypotheses of Lemma \ref{MatrixReznick}. Hence there are $k\in\N_0$ and $Q\in\R[y]^{m\times d}$ with $\|y\|^{2k}H_a=Q^TQ$. With $v=(s_{d-1},\ldots,s_0)^T$ as in Proposition \ref{HermiteForm}, we get
    \begin{multline*}\|y\|^{2k}L_p((1+a^Tx)(y^Tx-z_1)^2\cdots(y^Tx-z_{d-1})^2)\\=\|y\|^{2k}v^TH_av=(Qv)^T(Qv)\in\textstyle\sum\R[y,z]^2,\end{multline*}
    so $a\in R_k(p)$.
\end{proof}

\begin{rem}
    The previous theorem is similar to a result in a remark on page 232 in \cite{kpv}. See Subsection \ref{sectionKPV} for more details on this similar approach.
\end{rem}

We will see a possible application of Proposition \ref{HermiteSOSImpliesTighteningExact} in the following subsection. For this application we need another way of counting real zeros of univariate polynomials. This can be done by using the \textit{Bézout matrix}, which is congruent to the Hermite matrix.
\begin{df}\label{dfBezoutMatrix}
    For univariate polynomials, the following construction can be found in \cite[Section 2.1]{kn} and dates back to the 19th century. We need a multivariate version of \cite{kn}. Let $p\in\R[x]$ be a polynomial of degree $d$ with $p(0)\neq0$ and let $q\in\R[x]$ be a polynomial of degree at most $d$. Denote by $P\in\R[t,x]$ and $Q\in\R[t,x]$ the homogenizations of $p$ and $q$, respectively. Then
    \begin{equation*}
        \frac{P(t,x)Q(s,x)-P(s,x)Q(t,x)}{t-s}
    \end{equation*}
    is a polynomial of degree at most $\deg(p)+\deg(q)-2$ in $(s,t)$. Denote by $c_{i,j}\in\R[x]$ the coefficient of $t^{i-1}s^{j-1}$ of this polynomial. The matrix $(c_{i,j})_{i,j=1,\ldots,d}$ is called \emph{Bézout matrix} of $p$ and $q$ and is denoted by $B(p,q)$ or $B(P,Q)$. We have
    \begin{equation}
        (1,t,\ldots,t^{d-1})B(P,Q)(1,t,\ldots,t^{d-1})^T=(\partial_tP)Q-P\partial_tQ.\label{eqWronzkianBezoutConnection}
    \end{equation}

    Consider the rational function $\frac{Q}{P}$ and consider its expansion as a formal power series at infinity
    \begin{equation*}
        \frac{Q}{P}=s_{-1}+s_0t^{-1}+s_1t^{-2}+\ldots.
    \end{equation*}
    Then, we define the Hankel matrix $H(p,q):=H(P,Q):=(s_{i+j-2})_{i,j=1,\ldots,d}$.
    If $P$ is already homogeneous of degree $d$ and $Q$ is already homogeneous of degree at most $d$, the same formulas define $B(P,Q)$ and $H(P,Q)$ directly.
\end{df}

\begin{rem}
    The Hankel matrix $H(P,\partial_tP)$ is the Hermite matrix of Definition \ref{dfHermiteMatrix}.
\end{rem}

\begin{pro}[{\cite[Equations 5 and 6]{kn}}]\label{propHermiteBezoutCongruent}
    Let $d\in\N$ and let $p\in\R[x]$ satisfy $p(0)\ne0$ and $\deg p\le d$. Write $p=p_0+\ldots+p_d$ with $p_k\in\R[x]_k$, setting $p_k:=0$ for $k<0$ and for $k>\deg p$. Let $q\in\R[x]$ have degree at most $d$, and set $S:=(p_{d-i-j+1})_{i,j=1,\ldots,d}$. The matrix $S$ is symmetric, and $\det(S)\in\{\pm p_0^d\}$. If $P$ and $Q$ are the degree-$d$ homogenizations of $p$ and $q$, then
    \begin{equation*}
        B(p,q)=SH(p,q)S.
    \end{equation*}
    The same identity holds when a homogeneous $P\in\R[t,x]$ of degree $d$ and a homogeneous $Q\in\R[t,x]$ of degree at most $d$ are given directly. Set $p(x):=P(1,x)$ and build $S$ from the coefficients of $p$ as above; then $B(P,Q)=SH(P,Q)S$ (see Definition \ref{dfBezoutMatrix}).
\end{pro}

\begin{lem}[Jacobi's formula]\label{Jacobi}
    Let $A_1,\ldots,A_n\in\C^{d\times d}$ be hermitian matrices. Let $a\in\R^n$. Then, denoting $tI_d+x_1A_1+\ldots+x_nA_n$ by $A$,
    \begin{equation*}
        D_{(1,a)}\det(A)=\det(A)\tr(A^{-1}A(1,a)).
    \end{equation*}
\end{lem}
\begin{proof}
    This follows immediately from the well-known Jacobi's formula about the derivative of the determinant of a matrix function.
\end{proof}

\begin{lem}\label{geomSeries}
    Let $A\in\R[x]^{d\times d}$ be a matrix. Then the inverse of $tI-A$ in the ring $\R[x]((t^{-1}))^{d\times d}$ of matrices over the formal Laurent series in $t^{-1}$ is
    \begin{equation*}
        \sum_{i=0}^\infty A^it^{-i-1}.
    \end{equation*}
\end{lem}
\begin{proof}
    Geometric series.
\end{proof}

\begin{pro}\label{HermiteBezoutProposition}
    Let $p\in\R[x]$ be a degree $d$ real zero polynomial with $p(0)=1$ and write $p=\sum_{i=0}^dp_i$, where $p_i\in\R[x]_i$. Let $a\in\R^n$. Set
    \begin{equation*}
        H:=\left(L_p\left((1+a^Tx)(-y^Tx)^{i+j-2}\right)\right)_{i,j=1,\ldots,d}\in\R[y]^{d\times d}
    \end{equation*}
    and $S:=(p_{d-i-j+1})_{i,j=1,\ldots,d}$, where $p_i=0$ if $i<0$. Write $H(x)$ for the matrix in $\R[x]^{d\times d}$ obtained by substituting $x$ for $y$. Then
    \begin{equation*}
        SH(x)S=B(P,D_{(1,a)}P),
    \end{equation*}
    where $P\in\R[t,x]$ is the homogenization of $p$. In particular,
    \begin{equation*}
        D_{(1,a)}PD_{(1,0)}P-PD_{(1,a)}D_{(1,0)}P=(1,\ldots,t^{d-1})SH(x)S(1,\ldots,t^{d-1})^T.
    \end{equation*}
\end{pro}
\begin{proof}
    Suppose first that we are given a determinantal representation
    \begin{equation*}
        p=\det(I_d+A),\qquad A=x_1A_1+\ldots+x_nA_n,
    \end{equation*}
    with hermitian $A_1,\ldots,A_n\in\C^{d\times d}$. Then $P=\det(tI_d+A)$, and $H(x)_{ij}=\tr((I+A(a))(-A)^{i+j-2})$ by Proposition \ref{LFormOfRidgePolyTimesLinearPoly}.
    By Lemmata \ref{Jacobi} and \ref{geomSeries}, we get
    \begin{align*}
        \frac{D_{(1,a)}P}{P}&=\tr\left((I+A(a))(tI+A)^{-1}\right)\\
        &=\tr\left((I+A(a))\sum_{i=0}^\infty (-A)^it^{-i-1}\right)\\
        &=\sum_{i=0}^\infty\tr\left((I+A(a))(-A)^i\right)t^{-i-1}.
    \end{align*}
    Thus, the Hankel matrix $H(P,D_{(1,a)}P)$ is $H(x)$. Since $D_{(1,a)}P$ is homogeneous of degree $d-1$, Proposition \ref{propHermiteBezoutCongruent} gives $SH(x)S=B(P,D_{(1,a)}P)$.

    In general, it suffices to prove the identity after evaluating $x$ at an arbitrary point $b\in\R^n$, since both sides are matrices of polynomials in $x$. Both the Bézout matrix and the matrix $S$ are compatible with this evaluation, and $(H(b))_{ij}=L_p((1+a^Tx)(-b^Tx)^{i+j-2})$. Choose hermitian $A,B\in\C^{d\times d}$ with $p(sa+tb)=\det(I_d+sA+tB)$ as in Proposition \ref{LFormOfRidgePolyTimesLinearPoly}. Then $P(t,b)=\det(tI_d+B)$, $(D_{(1,a)}P)(t,b)=\det(tI_d+B)\tr((tI_d+B)^{-1}(I_d+A))$ and $(H(b))_{ij}=\tr((I_d+A)(-B)^{i+j-2})$, so the computation above applies verbatim with $A$ in place of $A(a)$ and $B$ in place of $A$. This remains valid when $p_d(b)=0$: the constant term of $t\mapsto P(t,b)$ may vanish, and Proposition \ref{propHermiteBezoutCongruent} does not require it to be nonzero.
\end{proof}

\subsection{Elementary symmetric polynomials}\label{ssElSym}

\begin{notation}
    We denote by $\binom{[n]}{d}$ the set of $d$-element subsets of $[n]:=\{1,\ldots,n\}$. For $S\subseteq\{1,\ldots,n\}$, we write $x^S:=\prod_{i\in S}x_i$. For $d\in\{0,\ldots,n\}$, we write $s_d$ for the $d$th elementary symmetric polynomial $\sum_{S\in\binom{[n]}{d}}x^S$.
\end{notation}

\begin{thm}[{\cite[Theorem 5.5 and Remark 5.7]{kum3}, see also \cite{bra2}}]\label{KummerLMIRepElsym}
    Let $d\in\{1,\ldots,n\}$. Then the hyperbolicity cone $C(s_d,\1)$ is given by the linear matrix inequality
    \begin{equation*}
        B(a)\succeq0\quad(a\in\R^n),
    \end{equation*}
    where
    \begin{equation*}
        B:=\left(\frac{1}{d-\#(S\cap T)}\sum_{k\notin S\cup T}x_k\right)_{S,T\in\binom{[n]}{d-1}}.
    \end{equation*}
    Moreover, we have
    \begin{equation}\label{SDPSolution}
        Bm=s_d\1,
    \end{equation}
    where $m=(x^S)_{S\in\binom{[n]}{d-1}}$ is the vector of all multi-affine monomials of degree $d-1$ in $x$.
\end{thm}

\begin{thm}\label{ElsymBezoutSOS}
    Let $d\in\{1,\ldots,n\}$ and $a\in\R^n$. The Bézout matrix of $s_d(x+t\1)$ and $D_{(1,a)}s_d(x+t\1)$ is of the form
    \begin{equation*}
        B(s_d(x+t\1),D_{(1,a)}s_d(x+t\1))=(n-d+1)Q^TB(a+\1)Q,
    \end{equation*}
    where $B$ is the matrix of the previous theorem and $Q$ is the matrix
    \begin{equation*}
        \left(s_{d-j}(x_S)\right)_{S\in\binom{[n]}{d-1},j\in\{1,\ldots,d\}}.
    \end{equation*}
\end{thm}
\begin{proof}
    The following proof shares a lot of similarities with a construction from \cite[p. 261]{knp} (see also the proof of Theorem 6.2 in \cite{kum3}).
    
    Since the map $(f,g)\mapsto B(f,g)$ is bilinear and $B$ is a linear matrix polynomial, we prove the claim by showing the following two statements
    \begin{enumerate}[(a)]
        \item $B(s_d(x+t\1),\partial_ts_d(x+t\1))=(n-d+1)Q^TB(\1)Q$,
        \item $B(s_d(x+t\1),\partial_{x_1}s_d(x+t\1))=(n-d+1)Q^TB(e_1)Q$.
    \end{enumerate}
    
    To prove part (a), we use Equation (\ref{SDPSolution}) from the previous theorem to get
    \begin{align*}
        B(x+s\1)m(x+s\1)=s_d(x+s\1)\1,\\
        m(x+t\1)^TB(x+t\1)=s_d(x+t\1)\1^T.
    \end{align*}
    Now, we multiply by $m(x+t\1)^T$ from the left and by $m(x+s\1)$ from the right, respectively. This, together with the fact that $m^T\1=s_{d-1}$, yields
    \begin{align*}
        m(x+t\1)^TB(x+s\1)m(x+s\1)=s_d(x+s\1)s_{d-1}(x+t\1),\\
        m(x+t\1)^TB(x+t\1)m(x+s\1)=s_d(x+t\1)s_{d-1}(x+s\1).
    \end{align*}
    Subtracting the second from the first equation, we get
    \begin{multline*}
        (s-t)m(x+t\1)^TB(\1)m(x+s\1)\\
        =s_d(x+s\1)s_{d-1}(x+t\1)-s_d(x+t\1)s_{d-1}(x+s\1).\qquad(\ast)
    \end{multline*}
    Notice that $(n-d+1)s_{d-1}(x+t\1)=\partial_ts_d(x+t\1)$ and $Q(1,\ldots,t^{d-1})^T=m(x+t\1)$. Thus, (a) follows from $(\ast)$.

    For (b), consider the polynomial
    \begin{equation*}
        s_d(x+s\1)\partial_{x_1}s_d(x+t\1)-s_d(x+t\1)\partial_{x_1}s_d(x+s\1)
    \end{equation*}
    and note that it does not depend on the variable $x_1$. Using that $s_d$ is multi-affine, we get
    \begin{align*}
        &s_d(x+s\1)\partial_{x_1}s_d(x+t\1)-s_d(x+t\1)\partial_{x_1}s_d(x+s\1)\\
        &=s_d(s,x_2+s,\ldots,x_n+s)s_{d-1}(x_2+t,\ldots,x_n+t)\\
        &-s_d(t,x_2+t,\ldots,x_n+t)s_{d-1}(x_2+s,\ldots,x_n+s)\\
        &=(s-t)s_{d-1}(x_2+s,\ldots,x_n+s)s_{d-1}(x_2+t,\ldots,x_n+t)\\
        &+s_d(x_2+s,\ldots,x_n+s)s_{d-1}(x_2+t,\ldots,x_n+t)\\
        &-s_d(x_2+t,\ldots,x_n+t)s_{d-1}(x_2+s,\ldots,x_n+s)\\
        &\stackrel{(a)}{=}(s-t)\sum_{\substack{\#S=\#T=d-1\\1\notin S\cup T}}(x+s\1)^S(x+t\1)^T\\
        &+(s-t)\sum_{\substack{\#S=\#T=d-1\\1\notin S\cup T}}\frac{1}{d-\#(S\cap T)}((n-1)-\#(S\cup T))(x+s\1)^S(x+t\1)^T\\
        &\stackrel{(\ast\ast)}{=}(s-t)\sum_{\substack{\#S=\#T=d-1\\1\notin S\cup T}}\frac{n-d+1}{d-\#(S\cap T)}(x+s\1)^S(x+t\1)^T\\
        &=(s-t)(n-d+1)\underbrace{(1,\ldots,s^{d-1})Q^T}_{((x+s\1)^S)_{S\in\binom{[n]}{d-1}}}B(e_1)\underbrace{Q(1,\ldots,t^{d-1})^T}_{((x+t\1)^T)_{T\in\binom{[n]}{d-1}}}.
    \end{align*}
    Here, $(\ast\ast)$ follows from
    \begin{equation*}
        n-1-\#(S\cup T)+d-\#(S\cap T)=n+d-1-(\#S+\#T)=n-1+d-(2d-2)=n-d+1
    \end{equation*}
    for all $S,T\in\binom{[n]}{d-1}$.
\end{proof}

\begin{thm}
    Let $d\in\{1,\ldots,n\}$, then
    \begin{equation*}
        R_0(s_d(x+\1))=C(s_d(x+\1)).
    \end{equation*}
\end{thm}
\begin{proof}
    Let $p:=\frac{s_d(x+\1)}{\binom nd}$, so that $p(0)=1$, $C(p)=C(s_d(x+\1))$ and $R_0(p)=R_0(s_d(x+\1))$ (the linear form $L_p$ does not change under scaling of $p$). Since $s_d$ is hyperbolic with respect to $\1$, $p$ is real zero by Remark \ref{nonzeroate}(d), and $a\in C(p)$ if and only if $a+\1\in C(s_d,\1)$. Let $a\in C(p)$. By Theorem \ref{KummerLMIRepElsym}, $B(a+\1)\succeq0$, so $B(a+\1)=R^TR$ for some real matrix $R$, and Theorem \ref{ElsymBezoutSOS} together with the bilinearity of the Bézout matrix gives
    \[B(P,D_{(1,a)}P)=\tbinom nd^{-2}B(s_d(x+t\1),D_{(1,a)}s_d(x+t\1))=Q'^TQ'\]
    with $Q':=\tbinom nd^{-1}\sqrt{n-d+1}\,RQ\in\R[x]^{k\times d}$, where $P=\frac{s_d(x+t\1)}{\binom nd}$ is the homogenization of $p$. By Proposition \ref{HermiteBezoutProposition}, $SH(x)S=Q'^TQ'$, where $S\in\R[x]^{d\times d}$ has determinant $\pm1$ by Proposition \ref{propHermiteBezoutCongruent} and is therefore invertible over $\R[x]$. Hence $H(x)=(Q'S^{-1})^T(Q'S^{-1})$ with $Q'S^{-1}\in\R[x]^{k\times d}$, and substituting $y$ for $x$ gives $H_a=Q_a^TQ_a$ with $Q_a\in\R[y]^{k\times d}$. Proposition \ref{HermiteSOSImpliesTighteningExact} yields the claim.
\end{proof}

\begin{rem}
    Brändén \cite{bra2} gives a spectrahedral representation of $C(s_d(x+\1))$ (see also Theorem \ref{KummerLMIRepElsym}).
\end{rem}

\subsection{Connection to the tightening of Kummer, Plaumann, and Vinzant}\label{sectionKPV}
In \cite{kpv}, Kummer, Plaumann, and Vinzant obtained a similar tightening, but for hyperbolic polynomials instead of real zero polynomials.
They realized that the hyperbolicity cone of a square-free hyperbolic polynomial is a linear slice of the cone of non-negative polynomials, see Theorem 3.1 in \cite{kpv}. Namely, if $p\in\R[x]$ is hyperbolic with respect to $e\in\R^n$ and square-free, then
\begin{equation*}
    C(p,e)=\{a\in\R^n\mid D_apD_ep-pD_aD_ep\geq0\text{ on }\R^n\}.
\end{equation*}
In Section 3, they tighten the non-negativity constraint to get a hierarchy of tightenings of the hyperbolicity cone. If the hyperbolic polynomial is smooth, then they show that the union over all the tightenings captures the whole interior of the hyperbolicity cone (see page 232 in \cite{kpv}).

Proposition \ref{hbrz} gives a way of translating hyperbolic to real zero polynomials. Using this, one can observe that their tightening is related to our tightening via the congruence of the Hankel matrix $H(P,Q)$ and the Bézout matrix $B(P,Q)$ from Proposition \ref{propHermiteBezoutCongruent} and the fact from Equation (\ref{eqWronzkianBezoutConnection}) that
\begin{equation*}
    (1,t,\ldots,t^{d-1})B(P,Q)(1,t,\ldots,t^{d-1})^T=(\partial_tP)Q-P\partial_tQ,
\end{equation*}
using the notation from Definition \ref{dfBezoutMatrix}.

\section*{Declaration of AI usage}
Artificial intelligence (AI) was used for the final polishing of this article.
Indeed, all results, and the arguments for them, come from \cite{sch} and \cite{saw},
against which the article can be checked. There is only one exception: In the first author's
thesis \cite{saw}, the alternative multiplier from Remark \ref{remMultiplier} was used which
lead to an
error in the proof of \cite[Thm. 6.0.16]{saw} which was corrected by inserting
the local Lemma \ref{MatrixReznick} whose proof was found by AI assistance.

\end{document}